\documentclass[10pt, reqno, english]{amsart}
\usepackage{enumitem,fullpage,amssymb,thmtools,xcolor,babel}
\usepackage{thmtools, thm-restate}

\newtheorem{theorem}{Theorem}
\newtheorem{lemma}[theorem]{Lemma}

\newtheorem{observation}[theorem]{Observation}

\numberwithin{claimcounter}{theorem}
\newtheorem{claim}[theorem]{Claim}

\newtheorem*{claim*}{Claim}

\newcommand{\subpattern}{\preceq}
\newcommand{\subpatternr}{\succeq}

\theoremstyle{definition}

\newcommand{\mS}{\bar{S}}
\newcommand{\mR}{\bar{R}}

\author[F1.Molla]{Theodore Molla}
\author[F2.Nelson]{Corey Nelson}
\title{The Basis of Foot-Sortable Sock Orderings}

\thanks{
\\ \indent TM: Department of Mathematics and Statistics, University of South Florida, Tampa, FL. Email: \texttt{molla@usf.edu}.   Research supported in part by NSF grants DMS-1800761 and DMS-2154313.
\\ \indent CN: Department of Mathematics, University of Tennessee, Knoxville, TN. Email: \texttt{cnelso74@vols.utk.edu}. The majority of this research effort was performed while this author was an undergraduate student at the University of South Florida.
}

\begin{document}

\begin{abstract}
  Defant and Kravitz considered the following problem:
  Suppose that, to the right of a foot, there is a line of colored socks that needs to be sorted.
  However, at any point in time, 
  one can only either place the leftmost sock to the right of the foot onto the foot (stack)
  or remove the outermost sock on the foot and make it the rightmost sock to the left of the foot (unstack).
  In this paper, we explicitly describe all minimal initial sock orderings that are unsortable.
\end{abstract}

\maketitle

\section{Introduction}
Consider the following problem, which was introduced by 
Defant and Kravitz \cite{defant2022footsorting}.
Starting with an arbitrary line of colored socks to the right of a foot, 
one aims to move the socks to the left of the foot so that the ordering is \emph{sorted} by color; that is, the output ordering consists of blocks of socks so that all socks of the same color appear consecutively. At any point in time, however, only one of the following two actions can be performed.
\begin{enumerate}[left=40pt]
  \item[(stack)] Take the leftmost sock on the right and place it on the foot (potentially over any socks
    that are already on the foot).
  \item[(unstack)] Remove the outermost sock from the foot and place it so that it becomes the rightmost sock on the left.
\end{enumerate}
The initial ordering of socks on the right is a \textit{sock ordering} and 
a sock ordering is \textit{foot-sortable} (or simply \textit{sortable}) if using only the two specified operations, one can produce
a sorted ordering of the socks on the left. Otherwise, it is \textit{unsortable}. Throughout, we represent sock orderings as strings in the obvious way; left-to-right in the string corresponds
to left-to-right in the sock ordering.
For example, if \(r\), \(g\), and \(b\) refer to socks of color
red, green, and blue, respectively,
the sock ordering \(rgbrgb\) is sortable (to \(rrbbgg\) on the left) while the sock ordering
\(rgbgrgb\) is unsortable.

We will use the following two definitions to describe our result.
Two sock orderings are \textit{equivalent} if
one can be obtained from the other by a bijective mapping of the colors.
So, for example, we think of \(rgbgrgb\) as equivalent to \(gbrbgbr\).
We call a sock ordering \textit{minimally unsortable} if it is unsortable and
removing any sock yields a sortable sock ordering. 
For example, \(rgbgrgb\) is minimally unsortable, because it is unsortable
and the sock orderings
\(gbgrgb\),
\(rbgrgb\),
\(rggrgb\),
\(rgbrgb\),
\(rgbggb\),
\(rgbgrb\), and
\(rgbgrg\)
are each sortable
(to \(gggrbb\), \(rrggbb\), \(rrgggb\), \(rrbbgg\), \(rgggbb\), \(ggbbrr\), and \(bgggrr\), respectively). 

Modulo equivalence, 
Defant and Kravitz 
called the unique set of minimally unsortable sock orderings 
the \textit{basis of foot-sortable sock orderings} 
\cite{defant2022footsorting}.
In this paper, we explicitly described this set.
This description appears in Section~\ref{sec:basis},
and the rest of the paper is devoted to proving that this set is indeed the basis of foot-sortable
sock orderings.
(For reference, \(rgbgrbg\) is equivalent to the ordering \(\mathcal{T}_6\) that is defined
in Section~\ref{sec:basis}.)

\subsection{Prior Work}

As observed by Defant and Kravitz \cite{defant2022footsorting}, this problem fits into the rich line of research
dealing with sorting procedures in which the employed data structures are restricted.
For example, sorting permutations with a stack was considered by Knuth \cite{knuth97}.
Defant and Kravitz also considered several variants of this problem, 
including employing multiple stacks (in particular two stacks or two feet) and restricting the number of socks of a given color.
Of particular interest with respect to socks is the case when there are at most two socks of every color.
Independent of our work, Yu explicitly described the basis for 
foot-sortable orderings in which each color appears
at most twice \cite{yu2023deciding}.
In doing so, Yu determined that this basis is infinite, and this 
implies that the basis of foot-sortable sock orderings is also infinite. 
This answered a question of Defant and Kravitz which our work also addresses.
Also independent of our work, Yu devised a fast deterministic algorithm
that decides if a sock ordering of length \(N\) is sortable in time \(O(N \log N)\).
Provided that the ordering is sortable, this algorithm can also produce a valid sorting.
This problem was also investigated by Xia, who considered a deterministic algorithm that forces the stack to remain sorted \cite{xia2023deterministic}.

\section{Preliminary discussion and a description of the basis.}

\subsection{Additional Definitions and Notation}

Let \(X\) and \(Y\) be sock orderings. 
We say that \(Y\) is a \textit{subordering} of \(X\) and write \(Y \subseteq X\) (or \(X \supseteq Y\)) if 
deleting zero or more socks from \(X\) yields \(Y\) exactly.
The subordering \(Y \subseteq X\) is \textit{proper} if \(X \neq Y\).
So, an unsortable sock ordering is minimally unsortable if every proper subordering is sortable.
We write \(a \in X\) if the color \(a\) appears in \(X\)
and we let \(X - a\) be the subordering of \(X\) obtained by deleting all socks
of color \(a\) from \(X\). We use \(\emptyset\) to denote the empty ordering and
\(a\) to denote the ordering consisting of a single sock of color \(a\).
We let \(XY\) be the concatenation of the orderings \(X\) and \(Y\) and
we write \(X \cong Y\) when \(X\) and \(Y\) are equivalent.
We call \(Y\) a \textit{subpattern} of \(X\) and write \(Y \subpattern X\) (or \(X \subpatternr Y\)) if there exists 
a sock ordering \(Z\) such that \(Z \subseteq X\) and \(Z \cong Y\).
Otherwise, we say that \(X\) \textit{avoids} \(Y\).
Furthermore, \(X\) avoids a collection of sock orderings if it avoids every ordering in the collection.

Throughout, we will refer to the sorting process using primarily stack-centered language. 
We denote a particular state of the foot-sorting process as \((S,R)\). 
Here \(S\) represents the contents of the stack, with left-to-right in \(S\) 
corresponding with bottom-to-top in the stack (or inside-to-outside on the foot), 
while \(R\) represents what remains on the right. 
We say that the state \((S,R)\) is \textit{foot-sortable} (or just \textit{sortable}) if 
we can apply the operations stack and unstack so that all of the socks in \(SR\) 
are arranged by color to the left of the foot.
Note that if the stack \(S\) is not sorted,
then \((S,R)\) is unsortable.
As such, we implicitly assume throughout that the stack \(S\) is sorted.
We call a state \((S', R')\) a \textit{substate} of \((S, R)\) if \(S' \subseteq S\) and \(R' \subseteq R\),
and it is proper if at least one of the two containment relations is strict.

Instead of dealing with the two operations stack and unstack, we work almost exclusively 
at a slightly higher level. At each point, we are mainly concerned with the next
sock color that will be moved completely to the left, and we say
that a color is \textit{sortable} if it is possible to do this while
keeping the stack sorted. Of great importance in the context of moving a color to the left is the concept of a \emph{sandwich}: a configuration of (possibly non-consecutive) socks taking the form \(bcb\), where \(b\) and 
\(c\) are distinct colors. We refer to a sandwich of the form \(bcb\) as either a
\(b\)-sandwich or a \(c\)-sandwich.
More explicitly, given a state \((S,R)\), a color \(a \in SR\) is sortable if the following two
conditions hold.
\begin{enumerate}
  \item For any \(b\in S\) which is distinct from \(a\), \(ab\not\subseteq S\).  (That is,
    if \(a\) appears on the stack, it must be the top color on the stack.)
  \item If \(a\) appears in \(R\), then \(bcba\not\subseteq SR\) for any colors \(b,c\in SR\) such that \(a\), \(b\), and \(c\) are distinct. 
\end{enumerate}
If a color \(a \in R\) does not meet the second condition for colors \(b,c\) in \(SR\), we say that the color \(a\)
is blocked by the sandwich \(bcb\).  
Given a state \((S,R)\), a color \(a \in SR\) is a
\emph{good sortable color} if \(a\) is sortable in \((S,R)\) and what remains after sorting \(a\) 
in \((S,R)\) is foot-sortable.
If no color is sortable in \((S,R)\), then we say that \((S,R)\) is \textit{terminal}. 
Given a sock ordering \(X\), a color \(a \in X\) is sortable (resp.\ unsortable) if it is sortable (resp.\ unsortable) 
in \((\emptyset, X)\), and a sock ordering \(X\) is \textit{trivially unsortable} if no color in \(X\) is sortable.
In other words, \(X\) is trivially unsortable if the state \((\emptyset, X)\) is terminal.  

Note that sorting the color \(a\) transitions the state \((S, R)\) to
the state \((S'R', R'')\) where
\(S' = S - a\) (so \(S' = S\) if \(a\) is not on the stack \(S\)); 
and, when \(a \in R\), \(R'\) is the string which precedes the final \(a\) in \(R\) minus the color \(a\)
and \(R''\) is the string which follows the final \(a\) in \(R\).
When \(a\) does not appear in \(R\), we have \(R' = \emptyset\) and \(R'' = R\).

A sequence of distinct colors \(\theta = a_1, \dotsc, a_\ell\)
is a \textit{sorting sequence} for \( (S, R) \) if there are
orderings \(S_1, \dotsc, S_\ell\) and \(R_1,\dotsc,R_\ell\) such that when 
\(S_0 := S\) and \(R_0 := R\), for every \(i \in [\ell]\), the color 
\(a_{i}\) is sortable with respect to the state \((S_{i-1}, R_{i-1})\) and sorting \(a_i\) in this state
leads to \((S_i, R_i)\). We define \(\theta(S,R) := (S_\ell, R_\ell)\)
and we say that \(\theta\) \textit{fully sorts} \(X\) if 
\(S_\ell = R_\ell = \emptyset\).
For a sock ordering \(X\), we let \(\theta(X) := \theta(\emptyset, X)\).

\subsection{Interlaced sequences}

As mentioned previously, the basis of foot-sortable sock orderings is infinite.
We will use the following definition to help describe the infinite classes.

Let \(T := (t_0, \dotsc, t_n)\) be a sequence of \(n+1\) colors. 
Define  \(I(t_0) := I(T)\) to be the empty ordering when \(n=0\), and when \(n \ge 1\), let
\[
I(T) := I(t_0, \dotsc, t_n) := I(t_0, \dotsc, t_{n-1}) t_{n}t_{n-1} = t_1t_0t_2t_1 \dotsc t_{n}t_{n-1}.
\]
\begin{observation}\label{obs:inf_unsort}
Let \(n \ge 0\) and  \(a, b, t_0, \dotsc, t_n\) be distinct colors.
For every \(0 \le i \le n\), 
if \(abt_i \subseteq S\) and 
\(I(t_i, \dotsc, t_n) at_n \subseteq R\), then 
the state \((S, R)\) is unsortable.
Furthermore, for any \(0 \le i \le n\), 
if \((S', R')\) is a proper substate of 
\( (abt_i, I(t_i, \dotsc, t_n) at_n) \), then \((S', R')\) is sortable.
\end{observation}
\begin{proof}
The proof is by induction on \(n - i\).
For the base case, we have \(i = n\), so 
\[
I(t_i,\dotsc,t_n) a t_n = a t_n
\] 
and it is not hard to see \((abt_n, at_n)\) is terminal and removing any sock from 
\((abt_n, at_n)\) makes the state sortable.
Therefore, both of the statements are true when \(i = n\).

For the induction step, first note that sorting \(t_i\) in 
\[
(abt_i, I(t_{i}, \dotsc, t_n) a t_n) =  
(abt_i, t_{i+1}t_iI(t_{i+1}, \dotsc, t_n) a t_n)
\]
leads to \( (abt_{i+1}, I(t_{i+1}, \dotsc, t_n) a t_n) \).
Note that when \(n - i = 1\) the only sortable color is \(t_i\),
so the first statement follows by induction.
When \(n - i \ge 2\), the only sortable colors are \(t_i\) and \(t_{i+1}\)
and sorting \(t_{i+1}\) in 
\[
(abt_i, I(t_{i}, \dotsc, t_n) a t_n) =  
(abt_i, t_{i+1}t_it_{i+2}t_{i+1}I(t_{i+2}, \dotsc, t_n) a t_n),
\]
leads to \( (abt_{i}t_it_{i+2}, I(t_{i+2}, \dotsc, t_n) a t_n)\).
Therefore, the first statement again follows by the induction hypothesis.

We will now prove the second statement when \(n - i \ge 1\).  
Recall that we are assuming \(S' \subseteq abt_i\) and 
\[
R' \subseteq I(t_i, \dotsc, t_n)at_n = t_{i+1}t_iI(t_{i+1}, \dotsc, t_n)at_n
\]
where at least one of the two, \(S'\) or \(R'\), is a proper subordering.
By possibly adding socks, we
also assume that either \(S' = abt_i\) or
\(R' = t_{i+1}t_iI(t_{i+1}, \dotsc, t_n)at_n \).
First assume that \(R' = t_{i+1}t_iI(t_{i+1}, \dotsc, t_n)at_n\),
so \(S' \neq abt_i\).
If \(b \in S'\) and \(t_i \notin S'\), then we sort \(b\) and then \(t_i\).
Otherwise, we only sort \(t_i\).
In either case, we reach \((S'', R'')\) where \(S''\) is a proper subordering of
\(abt_{t+1}\) and  \(R'' = I(t_{i+1}, \dotsc, t_n)a t_n\).
The conclusion then follows from the induction hypothesis.

Now, assume \(S' = abt_i\). So,
\(R' \neq t_{i+1}t_iI(t_{i+1}, \dotsc, t_n)at_n\).
If \(R'\) does not begin with \(t_{i+1}t_i\), then
we can sort \(t_i\) and then \(b\).
After possibly pushing \(t_{i+1}\) onto the stack, we
arrive at the state \((S'', R'')\) where \(S'' \subseteq at_{i+1}\) and 
\(R'' \subseteq I(t_{i+1}, \dotsc, t_n)at_n\).
In this case, since \(b \notin S''\), the conclusion follows from the induction hypothesis.
Otherwise, we have \(S' = abt_i\) and \(R' = t_{i+1}t_i R''\) where \(R''\) is a proper
subordering of \(I(t_{i+1}, \dotsc, t_n)at_n\).
Then, after sorting the color \(t_i\), we have \((abt_{i+1}, R'')\)
and the conclusion follows by the induction hypothesis.
\end{proof}

\subsection{The basis of foot-sortable sock orderings}\label{sec:basis}

The tables below describe the basis of foot-sortable sock orderings.
We use \(\Gamma\) to denote this set, so our main theorem is the following.
\begin{theorem}\label{thm:main}
If \(X\) is an unsortable sock ordering, then a subpattern of \(X\) is in \(\Gamma\).
\end{theorem}

With Theorem~\ref{thm:main}, to show that \(\Gamma\) is the basis of foot-sortable sock orderings
we also need to show that every ordering in \(\Gamma\) is minimally unsortable.
\medskip

The following table consists of the trivially unsortable orderings in \(\Gamma\).
\begin{center}
    \begin{tabular}{|c|c|c|c|}
    \hline
        \(\mathcal{T}_1\cong abacaba\) & \(\mathcal{T}_2\cong abacbab\) & \(\mathcal{T}_3\cong abcabca\) & \(\mathcal{T}_4\cong abcacba\) \\
    \hline
        \(\mathcal{T}_5\cong abcabac\) & \(\mathcal{T}_6\cong abcbacb\) & \(\mathcal{T}_7\cong abcbabc\) &  \\
    \hline
        \(\mathcal{T}_8\cong abcadcab\) & \(\mathcal{T}_9\cong abacdcba\) & \(\mathcal{T}_{10}\cong abacdcab\) & \\
    \hline
    \end{tabular}
\end{center}
It is not hard to verify that for each ordering listed above, every color is blocked by sandwich, so
the orderings are indeed trivially unsortable.
It is also not difficult to verify that whenever a sock is removed, 
a color becomes sortable.
This is enough to show that \(\mathcal{T}_1, \dotsc, \mathcal{T}_7\)
are minimally unsortable, because they each contain only three distinct colors, which means that
after a color is sorted in one of the suborderings only at most two colors remain, but  
three distinct colors are necessary for a color to be blocked by a sandwich.
Similarly, 
for \(\mathcal{T}_8\), \(\mathcal{T}_9\), or \(\mathcal{T}_{10}\), 
it can be checked that removing one sock
produces an ordering in which two colors can be successively sorted.
After this, only two distinct colors then remain, so the suborderings are sortable.
\medskip

All of the sock orderings in the next table have either \(atdarat\), \(atdrart\), or \(atdarbrt\) as a subordering.
Each of these three suborderings is sortable if one first sorts the color \(a\).
However, for these three suborderings, 
sorting the color \(d\) 
yields \((at, arat)\), \((at, rart)\), and \((at, arbrt)\), respectively,
and each of these three states are terminal. 
So, to construct the unsortable sock orderings in this table, either a new sock is inserted 
so that a \(d\)-sandwich blocks \(a\),
or \(ed\) is placed immediately after \(d\),
After this addition, the only sortable colors are \(d\) and if \(ed\) was added, \(e\).
This implies that the orderings are unsortable.
\begin{center}
    \begin{tabular}{|c||c||c|}
    \hline
        \(\mathcal{C}_{1,1}\cong datdarat\) & \(\mathcal{C}_{1,6}\cong atdtrart\) & \(\mathcal{C}_{1,10}\cong atdtarbrt\) \\
    \hline
        \(\mathcal{C}_{1,2}\cong atdtarat\) & \(\mathcal{C}_{1,7}\cong datdrart\) & \(\mathcal{C}_{1,11}\cong adtdarbrt\) \\
    \hline
        \(\mathcal{C}_{1,3}\cong atdatrat\) & \(\mathcal{C}_{1,8}\cong adtdrart\) & \(\mathcal{C}_{1,12}\cong datdarbrt\) \\
    \hline
        \(\mathcal{C}_{1,4}\cong adtdarat\) & \(\mathcal{C}_{1,9}\cong atdedrart\) & \(\mathcal{C}_{1,13}\cong atdedarbrt\) \\
    \hline
        \(\mathcal{C}_{1,5}\cong atdedarat\) &  &  \\
    \hline
    \end{tabular}
\end{center}

To see that these orderings are minimally unsortable,
first consider the suborderings of \(\mathcal{C}_{1,3}\cong atdatrat\) formed by removing a single sock:
\[ 
tdatrat, 
adatrat, 
atatrat,
atdtrat,
atdarat,
atdatat,
atdatrt,\text{ and } 
atdatra.
\] 
For \(adatrat\), \(atatrat\), and \(atdarat\), we have removed a sock from the \(d\)-sandwich that blocks
\(a\), so we can fully sort each of these three subordering by first sorting \(a\), and then sorting \(r\).
For the remaining suborderings 
(\(tdatrat\), \(atdtrat\), \(atdatat\), \(atdatrt\), and \(atdatra\)),
we can first sort \(d\) and then one additional color (\(a\), \(r\), \(t\), \(t\), and \(t\), respectively), 
leaving at most two colors, so the suborderings are sortable.
For every other subordering formed by removing one sock from an ordering in the table,
one can verify that one of the following two possibilities hold:
\begin{itemize}
  \item One of the socks in the \(d\)-sandwich that blocks \(a\) has been removed and 
  the subordering can then be fully sorted when one first sorts the color \(a\), or 
  \item sorting \(d\) (and then \(e\) if it appears) yields either \((S, R)\) or \((S, tR)\)
   where \((S, R)\) is a proper substate of one of the states 
   \((at, arat)\), \((at, rart)\), or \((at, arbrt)\), thereby implying 
   that the subordering is sortable.
\end{itemize}
\medskip

The next table consists of the sock orderings that have either 
\(abtdabt\) or \(abtdatb\) as a subordering.  
\begin{center}
    \begin{tabular}{|c||c||c|}
    \hline
        \(\mathcal{C}_{2,1}\cong abtdatab\) & \(\mathcal{C}_{2,4}\cong adbtdatb\) & \(\mathcal{C}_{2,6}\cong abtdedabt\) \\
    \hline
        \(\mathcal{C}_{2,2}\cong abtdtatb\) & \(\mathcal{C}_{2,5}\cong abdtdatb\) & \(\mathcal{C}_{2,7}\cong abtdedatb\) \\
    \hline
        \(\mathcal{C}_{2,3}\cong adbtdabt\) &  &  \\
    \hline
    \end{tabular}
\end{center}
Sorting \(d\) (and possibly \(e\)) first in these orderings 
yields \((abt,R)\) where \(at\subseteq R\). Since \((abt,at)\) is terminal and \(d\) and \(e\) are 
the only sortable colors in these orderings, the orderings in this table are unsortable.

To address minimality, note that each of the following holds for every orderings in the table:
\begin{itemize}
  \item If the final \(t\) is removed, then \(d\) is a good sortable color.
  \item If the first \(a\) or \(b\) is removed, then \(d\) is a good sortable color.
  \item If the final \(a\) or \(b\) is removed, then the color of the removed sock becomes a good sortable color.
  \item If \(d\) or \(e\) is removed, then \(a\) is a good sortable color.
  \item If there is only one sandwich that blocks \(a\), then removing a sock from that sandwich
  makes \(a\) a good sortable color.
\end{itemize}
This covers all possibilities except removing the middle \(a\) when there are three \(a\)'s or
the first \(t\) when it is not part of the only sandwich that blocks \(a\).
This leaves the following suborderings to consider:
\[
abtdtab \subseteq \mathcal{C}_{2,1},\quad
adbdabt \subseteq \mathcal{C}_{2,3},\quad
adbdatb \subseteq \mathcal{C}_{2,4},\quad
abdedabt \subseteq \mathcal{C}_{2,6},\quad\text{and}\quad
abdedatb \subseteq \mathcal{C}_{2,7},
\]
and \(d\) is a good sortable color in each of these three suborderings.
\medskip

For every \(n \ge 0\), let \(a\), \(b\), \(d\), \(e\), and \(t_0, \dotsc, t_n\) be distinct colors.
Let \(T_n := (t_0, \dotsc, t_n)\), \(E_n := I(T_n)at_n\)
and define \(t := t_0\). (Note that \(E_0 = at_0 = at\)).
We define the following five orderings which each contain \(abtdE_n\).
\begin{center}
    \begin{tabular}{|c||c||c||c||c|}
    \hline
    \(I_{1,n} := dabtdE_n\) &
    \(I_{2,n} := adbtdbE_n\) &
    \(I_{3,n} := abdtdbE_n\) &
    \(I_{4,n} := abtdtbE_n\) &
    \(I_{5,n} := abtdedbE_n\) \\\hline
\end{tabular}
\end{center}
By Observation~\ref{obs:inf_unsort}, the state \((S, R)\) is unsortable if 
\(abt \subseteq S\) and \(E_n \subseteq R\). 
Therefore, none of the orderings described below can be fully sorted by first sorting
\(d\) (or \(d\) and \(e\) if \(e\) is present).
Since this is the only option for \(I_{2,n}\), \(I_{3,n}\), \(I_{4,n}\), and \(I_{5,n}\), those
orderings are unsortable.
In \(I_{1,n}\), both the colors \(b\) and \(d\) are sortable, but sorting \(b\) leads to 
\((da, tdE_n)\) which is terminal, so \(I_{1,n}\) is also unsortable.

Let \((S, R)\) be a proper substate of \((abt, E_n)\).
Recall that Observation~\ref{obs:inf_unsort} implies that \((S, R)\) is sortable.
We also claim that \((S, bR)\) is also sortable.
To see this, first note that by adding socks we can assume that we either have \(S = abt\) or \(R = E_n\). 
Therefore, \(t_0 = t \in SR\).  If \(n \ge 1\), then \(t\) is a good sortable color in \((S, bR)\),
because \(t = t_0\) is sortable in \((abt, bE_n)\) and sorting \(t\) in the state \((S, bR)\)
yields a proper substate of \((abbt_1, I(t_1, t_2, \dotsc, t_n)a t_n)\) which is sortable
by  Observation~\ref{obs:inf_unsort}. 
If \(n = 0\), then \((S, R)\) is a proper substate of \((abt, at)\), and it is not hard to
then verify that \((S, bR)\) is then sortable.

The preceding argument implies that both \((S, R)\) and \((S, bR)\) are sortable when 
\((S, R)\) is a proper substate of \((abt, E_n)\).
This immediately implies that if a sock in the initial \(abt\) or the terminal \(E_n\) is removed from 
\(I_{1,n}\), \(I_{2,n}\), \(I_{3,n}\), and \(I_{5,n}\), then
the subordering is sortable.
For \(I_{4,n}\), the same statement also holds. 
To see this, first note that  
we can sort \(d\) and then \(b\) in \(abdtbE_n \subseteq I_{4,n}\) to reach the state \((at, E_n)\).
In all other cases when we delete a sock in the initial \(abt\) or the terminal \(E_n\) from \(I_{4,n}\),
we can sort \(d\) and push \(t\) onto the stack to reach either the state \((att, bE_n)\), the state \((btt, bE_n)\) or a state \((abtt, bR)\) where \(R\) is a proper subordering of \(E_n\).

For \(I_{2,n}\), \(I_{3,n}\), \(I_{4,n}\), and \(I_{5,n}\), the only other possible options
are to remove a sock of color \(d\), \(e\), or the second sock of color \(b\), or,
in \(I_{4,n}\), the second sock of color \(t\).
For the suborderings \(adbtbE_n \subseteq I_{2,n}\) and \(abdtbE_n \subseteq I_{3,n}\),
first sorting \(d\) and then \(b\) leads to the state 
\((at, E_n)\) which is sortable by Observation~\ref{obs:inf_unsort}.
For all other possibilities, if we first sort \(b\) and then \(d\) and \(e\) if they are present, 
we either enter the state  \((at, E_n)\),   
the state \((att, E_n)\), or the state \((at, tE_n)\) and each of these
states are sortable by Observation~\ref{obs:inf_unsort}.

For \(I_{1,n}\), the only other options are to remove one of the two socks of color \(d\).
If the first \(d\) is removed, then we can sort \(b\) followed by \(d\) which leaves
\((at, E_n)\), which is sortable by Observation~\ref{obs:inf_unsort}.
If the second \(d\) is removed, 
then we can again arrive at the state  \((at, E_n)\)
by first sorting \(d\), then sorting \(b\), and then pushing \(t\) onto the stack.

\section{Overview of the proof of Theorem~\ref{thm:main}}

Our proof of Theorem~\ref{thm:main} is inherently algorithmic, 
and we explicitly describe a deterministic recursive algorithm in Appendix~\ref{sec:algorithm}
which is based on the proof. 
Central to our proof is the following notion of a greedy algorithm.

\subsection{The greedy algorithm}\label{sec:greedy}

One's natural instinct may be to iteratively sort the first sortable color appearing in the sock ordering; that is, the sortable color appearing in \((S,R)\) whose final sock appears in \(SR\) before the final sock of any other sortable color in \((S,R)\). We refer to this process as the \emph{greedy algorithm}, and have named the corresponding process in our algorithm as such. 
In some instances, the greedy algorithm may fail.
In effect, it can trap certain socks in the stack which then leads to issues later on in the sorting process. 
For a simple example, consider the sock ordering \(abcdabc\); the greedy algorithm instructs us to sort the color \(d\) first since \(d\) is sortable in \(abcdabc\) and the final sock of color \(d\) appears before the final socks of any other sortable color (which, in this case, is only \(a\)). Sorting \(d\) yields \((abc,abc)\), which is terminal. On the other hand, ignoring the greedy algorithm and instead sorting the color \(a\) first yields \((bcd,bc)\). From this point, one may use the sorting sequence \(\theta=d,c,b\) and observe \(\theta(bcd,bc)=(\emptyset,\emptyset)\), and as such, the sorting sequence \(\theta'=a,d,c,b\) fully sorts \(abcdabc\). 
\medskip

\subsection{Proof overview}

Let \(X\) be unsortable sock ordering.
We prove that \(X\) does not avoid \(\Gamma\) by induction on the number of socks in \(\Gamma\).

We start by applying the greedy algorithm to \(X\). 
As \(X\) is unsortable, the algorithm returns a terminal state \((S, R)\) where at least one of \(S\) or \(R\) is nonempty.

First suppose that the stack \(S\) is empty or contains a color \(x\) such that
\((x, R)\) is terminal.
This implies that either \(R\) or \(xR\) is trivially unsortable. 
Since it is relatively straightforward to show that a trivially unsortable pattern does not avoid \(\Gamma\)
(c.f.\ Lemma~\ref{lem:greedytriv}), we can easily handle this case.
That is, we can show that \(R\) or \(xR\) does not avoid \(\Gamma\), so \(X\) does not avoid \(\Gamma\).

Therefore, we can assume that there are at least two colors on the stack \(S\) and for no color \(x \in S\)
is \((x, R)\) terminal.
Let \(t\) be top color on the stack and \(b\) the color 
that appears immediately below \(t\) on the stack. 
So, \(t \neq b\) and \(bt \subseteq S\).
Since the stack is not empty, the greedy algorithm must have sorted at least one color.
We use \(d\) to denote the last color sorted by the greedy algorithm before termination. 

One simplifying observation is that one of the following two possibilities is true 
(c.f.\ Lemma~\ref{lem:greedy4}):  
\begin{itemize}
    \item \((bt, R)\) is terminal, or
    \item there exists a color \(a\) on the stack, distinct from \(b\) or \(t\), such that \((abt, R)\) is terminal.
      (For technical reasons, in this case, we always assume \(a\) is the lowest such color on the stack.)
\end{itemize}

Suppose \((bt, R)\) is terminal.  From this and the assumption that for every \(x \in S\) 
the state \((x, R)\) is not terminal, we deduce that \(R\) must contain \(P\) where
\(P\) is one of a few possible subpatterns and \(b\) 
is the only good sortable color in the subordering \(btP\).
For example, one possibility is that \(R\) contains \(brbt\) where \(r\) is some new color, and
the only good sortable color in \(btbrbt\) is \(b\) (in fact, \(b\) is the only sortable color in \(btbrbt\)).
If \(b\) is blocked by a \(d\)-sandwich in \(X\),
then we can argue that a subpattern of \(X\) is contained in \(\Gamma\).
Otherwise, we can argue that either \(X - d\) is unsortable 
or that \(I_{1,n}\) is a subpattern of \(X\) for some \(n \ge 0\) (c.f. Lemma~\ref{lem:nodsandwich}).
Since \(I_{1,n} \in \Gamma\), we are done in the second case.
In the first case, the induction hypothesis implies that \(X - d\), and hence \(X\), contains a subpattern that appears in \(\Gamma\).

Now assume the other possibility holds. That is, we have \(abt \subseteq S\) such that \((abt, R)\) is terminal.
This case is somewhat more complicated, but we can sometimes argue in a very similar manner to the previous case. 
However, sometimes the terminal state \((S,R)\) does not give us enough information
about the original ordering \(X\).
For example, it could be that \(S = abt\) and \(R = at\).
When we are in this situation, we backtrack to immediately before the greedy algorithm pushed the
final sock of color \(b\) onto the stack.  We then have sufficient information to complete 
the proof by making an argument similar to the previous cases.

\section{Observations and Lemmas}

In the enumeration of our basis \(\Gamma\), the five infinite classes have a common structure. 
In some sense, the following lemma explains the origin of this commonality.
\begin{lemma}\label{lem:interlinking}
    Let \(X\) be a sock ordering that, for some \(n \ge 0\), has 
    distinct colors \(b\) and \(c\) where
    the final \(b\) precedes the final \(c\) in \(X\). 
    Define
    \(A\) and \(C\) to be the strings such that \(X = AbC\) where \(b \notin C\)\
    and \(c \in C\).
    Let \(\theta\) be a nonempty sorting sequence for \(X\) that sorts a color in \(C\)
    and suppose \(c\) is the color in \(C\) that is sorted first by \(\theta\).
    Taking \((S, R) := \theta(X)\), suppose \(b \in S\)
    and that once \(c\) is sorted, \(b\) never appears as the 
    top color on the stack.
    Also, suppose that \(t_0,\dots,t_n\) (in this order) are the \(n+1\) colors which 
    appear immediately above the color \(b\) in the stack after \(c\) is sorted.
    Then 
    \[
    I(c,t_0, \dotsc, t_n)R \subseteq C. 
    \]
    Furthermore, the colors \(t_0, \dotsc, t_n\) are distinct and disjoint from \(b\) and \(c\).
\end{lemma}
\begin{proof}\mbox{}
    We will prove the first statement by induction on \(n\). 
    
    For the base case, assume \(n = 0\), so
    \(t_0= t_n\) is the only color that appears directly above \(b\)
    after \(c\) is sorted.
    Since \(c\) follows the last \(b\) and \(b\) can never appear
    as the top color on the stack after \(c\) is sorted, the color \(t_0 = t_n\)
    must appear above \(b\) after \(c\) is sorted.
    Therefore, we have 
    \[
      C \supseteq t_0cR = I(c, t_0, \dotsc, t_n)R.
    \]

    Now, suppose \(n \ge 1\).
    Let \(d\) be the color that is sorted that placed the color \(t_n\) 
    directly above the color \(b\) on the stack.
    And let \(\theta_d\) be the initial sequence of \(\theta\) up to but not including \(d\).
    Let \((S', R') := \theta_d(X)\).
    Note, during the execution of \(\theta_d\), 
    the color \(b\) is never the top color on the 
    stack and that \(t_0, \dotsc, t_{n-1}\) is the sequence of colors
    that appear immediately above
    the color \(b\) after \(c\) is sorted.
    Therefore, by induction, we have 
    \[
    I(c, t_0, \dotsc, t_{n-1}) R' \subseteq C.
    \]
    Recall that \(t_{n-1}\) appears immediately above \(b\) 
    on the stack \(S'\), 
    but after \(d\) is sorted \(t_n\) is the color above \(b\) on the stack.
    Since \(t_{n-1} \neq t_n\), this implies that \(d = t_{n-1}\).
    Furthermore, we also have \(t_n t_{n-1} R \subseteq R'\).
    This completes the proof of the first statement because 
    \[
    C  \supseteq I(c, t_0, \dotsc, t_{n-1}) R' \supseteq
    I(c, t_0, \dotsc, t_{n-1}) t_{n}t_{n-1} R =
    I(c, t_0, \dotsc, t_{n-1},t_n) R.\]

    For the final statement,
    note that, for every \(i \in \{0, \dotsc, n\}\), 
    the color \(t_i \neq b\) by definition and \(t_i \neq c\) since it appears in the 
    stack after \(c\) is sorted.  
    Furthermore, the colors \(t_0, \dotsc, t_n\) must be \(n+1\) distinct colors because
    the only way for a color to removed from the stack is for it to be sorted. 
\end{proof}

We now present a few simple observations and lemmas that will be useful throughout.
\begin{lemma}\label{lem:firsttwo}
Let \(X\) be a sock ordering.  Suppose that \(x \in X\) is unsortable
and that \(x\) is one of the first two colors to appear in \(X\).
Then any color that appears after the final \(x\) in \(X\) is blocked by
an \(x\)-sandwich.
\end{lemma}
\begin{proof}
Let \(yzy\) be a sandwich that blocks \(x\) in \(X\).
Since \(x\) is one of the first two colors to appear, 
\(z\) cannot precede the first occurrence of \(x\).
Therefore, both \(y\) and \(z\) appear between the first and last occurrence of \(x\).
Therefore, any color that follows the final \(x\) in \(X\) is blocked by one of
the sandwiches \(xyx\) or \(xzx\).
\end{proof}

There is a natural ordering of the sandwiches that appear in a given sock ordering.
If \(aba\) and \(xyx\) are suborderings of a sock ordering \(X\) we say that
\(aba\) \textit{precedes} \(xyx\) if the final \(a\) in \(aba\) precedes the final \(x\) in \(xyx\);
or when \(a = x\) and the final socks in \(aba\) and \(xyx\) are the same,
the \(b\) in \(aba\) precedes the \(y\) in \(xyx\);
or when \(a = x\) and \(b = y\) and the last two sock in \(aba\) and \(xyx\) are the same,
the first \(a\) in \(aba\) precedes the first \(x\) in \(xyx\). A sandwich \(aba\) is the \emph{first sandwich} to appear in a sock ordering \(X\) if for all sandwiches \(xyx\) in \(X\) which are distinct from \(aba\), \(aba\) precedes \(xyx\).

\begin{observation}\label{obs:firstsandwich}
If \(aba\) is the first sandwich to appear in a sock ordering \(X\), then 
any color that is not \(a\) or \(b\) is either sortable in \(X\) or is blocked by the
sandwich \(aba\).
\end{observation}
\begin{proof}
Let \(x\) be an unsortable color in \(X\) that is neither \(a\) nor \(b\).
Because \(x\) is unsortable, it must follow a sandwich.
Since \(aba\) is the first sandwich to appear, \(x\) must follow \(aba\).
Therefore, \(x\) is blocked by \(aba\).
\end{proof}

\begin{lemma}\label{lem:modsorting}
    Let \((S, R)\) be a sorting state (so \(S\) is sorted) and assume that \(S\) has 
    at least two distinct colors.
    Let \(a \in S\) where \(a\) is not the top color in \(S\) and let \(S'\)
    be formed by removing the color \(a\) from \(S\).
    If \(x\) is unsortable in \((S, R)\) and \(x\) is sortable in \((S', R)\), then 
    \begin{itemize}
        \item \(x = a\), or
        \item \(ax \subseteq R\), or
        \item there exists a color \(y\) distinct from the colors \(a\) and \(x\) such that
        \(yx \subseteq R\) and
        \(x\) is the top color, \(a\) is immediately below the color \(x\),
        and \(y\) is immediately below the color \(a\) on the stack \(S\).
    \end{itemize}
\end{lemma}
\begin{proof}
   Assume \(x \neq a\).
   Because \(x\) is sortable in \((S', R)\), we either have \(x \in R\), or that
   \(x\) is the top color on \(S'\).
   Suppose \(x\) is the top color on \(S'\). 
   Because \(a\) is not the top color on \(S\), this implies that 
   \(x\) is the top color on \(S\).
   Since \(x\) is not sortable in \((S, R)\), we have \(x \in R\) and a sandwich blocks \(x\) in \(SR\).
   So, in all cases, we have \(x \in R\) and a sandwich blocks \(x\) in \(SR\).

   Let \(yzy\) be the first sandwich in \(SR\) that blocks \(x\).
   Since \(S\) is sorted, \(yzy\) is not a subordering of \(S\), so we have \(yx \subseteq R\).
   Because \(x\) is sortable in \(S'R\), the color \(a\) is either \(y\) or \(z\).
   If \(a\) is \(y\), then \(ax \subseteq R\), so assume \(a\) is \(z\).
   Since \(yzy = yay \) does not block \(x\) in \(S'R\), we must have \(ya \subseteq S\).
   Recall that, by assumption, 
   \(a\) is not the top color in \(S\), so let \(w\) be the color that
   appears directly above \(a\) in \(S\).
   Note that \(w \neq y\) because \(S\) is sorted.
   If \(w \neq x\), then \(ywy\) blocks \(x\) in \(S'R\), so we have \(w = x\).
   Since \(x\) is sortable in \((S', R)\) and \(x\) is on \(S'\), 
   the color \(x\) must be the top color on the stack \(S'\).
   Therefore, \(x\) is the top color on \(S\).
   
   We now have that \(x\) is the top color and \(a\) is the color immediately below
   \(x\) on \(S\) 
   and that \(yx \subseteq R\).
   Therefore, we only need to show that the color \(y\) is directly below 
   the color \(a\) in \(S\).
   This follows from the observation that if 
   there exists a color \(u\) that is distinct from \(y\), \(a\), or \(x\)
   that is between \(y\) and \(a\) in \(S\), then the sandwich \(yuy\) blocks
   \(x\) in \(S'R\), a contradiction.
\end{proof}
\medbreak

The following lemma (Lemma~\ref{lem:nodsandwich}) is a crucial piece of our proof.
Before we present its statement and proof, we provide the following 
informal discussion of the statement, the proof, and how it fits into the larger problem.

Let \(X\) be a sock ordering (either sortable or unsortable) 
and suppose that an attempted sorting fails because it trapped
a color on the stack.  That is, there is a color that we need to sort, but we cannot
because it is beneath the top element on the stack.  
Let \(a\) be a color we need to sort and let \(d\) be the last color sorted in our attempted.
The hypothesis of Lemma~\ref{lem:nodsandwich} is how we characterize this situation precisely.

When this happens, it might have been difficult to sort \(a\) before sorting \(d\) because
there is a \(d\)-sandwich that blocks \(a\) in the original ordering \(X\) (c.f.\ \ref{I}). 
In this case, we can often argue that there was essentially no way to avoid trapping \(a\) on the stack,
and that \(X\) is unsortable because it does not avoid 
\(\Gamma\) (c.f.\ Lemmas~\ref{lem:greedy3} and \ref{lem:greedy2}).

If there is no \(d\)-sandwich blocking \(a\) in \(X\), 
then it might be natural to assume that sorting the color \(d\) 
was not the problem with our attempt.  
That is, maybe we sorted the wrong color earlier on in the process.
With this in mind, it might be reasonable to recursively try to 
sort the ordering formed by removing \(d\) from \(X\), \(X - d\). 
If that fails, then the recursive assumption implies 
that \(X - d\) is itself unsortable, which further implies that the original ordering, \(X\), 
is unsortable (c.f.\ \ref{II}).
So, we only need to decide how to proceed when we can sort \(X - d\). 

Let \(\theta\) be any sorting sequence that fully sorts \(X - d\).
One possibility is that if we apply \(\theta\) to the original sequence \(X\), 
then at some point the color \(d\) appears only at the top of the stack (c.f.\ \ref{III}).
Note that this implies that \(X\) itself is sortable; if we sort \(d\) at this point, 
then the state matches the corresponding state reached while sorting \(X - d\).
(For the proof of Theorem~\ref{thm:main}, we could replace \ref{III} with the simpler statement 
``\(X\) is sortable'', but that would not imply that 
when an attempt to sort \(X - d\) succeeds, we necessarily have a sorting sequence that 
can be used to sort \(X\),  so it would not imply that our recursive algorithm is correct.)

So let us assume otherwise.  (That is, we assume that neither \ref{I}, \ref{II}, nor \ref{III} hold.)
This is how we begin the proof of the lemma.  
With this assumption, we can quickly argue that 
\(X = YDW\) where \(D\) contains all of the socks of color \(d\) 
and no sock color besides \(d\) appears in \(D\).
Let \(\theta\) be any sorting sequence that fully sorts \(X - d = YW\).
We then argue that \(\theta\)  can be used to sort \(X\) up to and including the color \(a\).
Furthermore, when the sequence \(\theta\) sorts \(a\), it must push \(d\) onto the stack.  
We then analyze the situation immediately before a sandwich would appear on the stack when sorting 
\(X\) with \(\theta\) (this must happen
because the fact that \ref{III} does not hold implies that \(\theta\) cannot be used to
sort every color in \(X\) except \(d\), since then \(d\) would be the only element on the stack).
Using Lemmas~\ref{lem:interlinking} and \ref{lem:modsorting}, we can then determine that 
\(I_{1,n} \in \Gamma\) is a subpattern of \(X\) where \(n+1\) is the number of distinct colors that appear
immediately above \(d\) on the stack after we sort the color \(a\) (c.f.\ \ref{IV}).

\begin{lemma}\label{lem:nodsandwich}
Let \(X\) be a sock ordering and let \(\overline{\theta}\) be a nonempty sorting sequence for \(X\). Define \((S,R):=\overline{\theta}(\emptyset,X)\) and let \(d\) be the last color in \(\overline{\theta}\). Suppose that there exists a color \(a\) distinct from \(d\) and suborderings \(S'\subseteq S\) and \(R'\subseteq R\) such that \(a\) is not the top color of \(S'\), \(a\) appears in \(R'\), and \(a\) is the only good sortable color in \(S'R'\).
Then, one of the following is true:
\begin{enumerate}[label=(\Roman*)]
\item\label{I} there is a \(d\)-sandwich that blocks \(a\) in \(X\);
\item\label{II} \(X - d\) is unsortable;
\item\label{III} 
  \(X - d\) is sortable and if a sequence \(\theta\) fully sorts \(X - d\), then
  when we attempt to sort \(X\) with the sequence \(\theta\), 
  there is a stage where the color \(d\) only appears as the top color on the stack; or
\item\label{IV}  for some \(n \ge 0\), \(I_{1,n}\) is a subpattern of \(X\).
\end{enumerate}
\end{lemma}
\begin{proof}
    Assume that \ref{I}, \ref{II}, and \ref{III} do not hold. 
    We will show that \ref{IV} must hold.
    Let \(x\) be the top color on \(S'\) and \(D \subseteq X\)
    be such that \(X=YDW\) where the first element of \(D\) is the first \(d\) in \(X\)
    and the final element of \(D\) is the final \(d\) in \(X\).
    Note that the subordering \(R\) must follow the last \(d\) in \(X\),
    so \(R \subseteq W\).
    Similarly, everything in \(S\) must precede the last \(d\) in \(X\), so
    \(S \subseteq YD\).
    Since \(R \subseteq W\) and \(a \in R' \subseteq R\), we have \(a \in W\), so the fact 
    there are no \(d\)-sandwiches 
    that block \(a\) in \(X\) implies that the only colors that could appear in \(D\)
    are \(a\) and \(d\).
    In particular, the color \(x\) does not appear in \(D\), 
    so the fact that \(a\) is beneath \(x\) on \(S'\) and 
    \(S' \subseteq S \subseteq YD\) implies that \(ax \subseteq Y\).
    Note that because \(d\) is sorted by \(\overline{\theta}\) but \(a\) and \(x\) are not, 
    there is no sandwich that consists only
    of \(a\) and \(x\) that blocks \(d\) in \(X\).
    Therefore, the color \(a\) does not appear in \(D\) 
    as otherwise the fact that \(ax \subseteq Y\) would imply
    the existence of a sandwich of the form \(axa\) before the final \(d\) in \(X\).
    As such, \(d\) is the only color that appears in \(D\)
    and we can assume without loss of generality that the color \(d\) only appears once in \(X\),
    so \(D = d\).

    By assumption, the statements \ref{II} and \ref{III} are both false, so
    there exists a sorting sequence \(\theta\) that fully sorts \(X - d = YW\) such
    that when we apply \(\theta\) to \(X = YDW = YdW\) the color \(d\) never appears 
    only as the top color on the stack.
    Let \(\theta_{a}\) be the list of initial colors in \(\theta\) up
    to, but not including, \(a\) (so \(\theta_{a}\) is empty if and only if \(\theta\) begins
    with \(a\)).
    Note that no color in \(S'R'\) could appear before 
    \(a\) in \(\theta\) because \(a\) is the only good sortable color in \(S'R' \subseteq YW\)
    and \(\theta\) fully sorts \(YW\).
    We claim that this implies that no color that appears after \(ax\) in \(YW\) is in \(\theta_a\).
    To see this, note that sorting such a color would put \(ax\) onto
    the stack, but that would make it impossible to sort \(a\) before sorting \(x\).
    In particular, because \(ax \subseteq Y\), this means that 
    no color that appears in \(W\) is in \(\theta_a\).
    Therefore, \(a\) is the color in \(W\) that is
    sorted first by \(\theta\),
    and \(\theta_a\) is a sorting sequence for \(X = YdW\).
    Furthermore,
    if we let \((S_1, R'_1) = \theta_a(YW)\), 
    then \((S_1, R_1) = \theta_a(YdW)\) where
    \(R'_1 = R_1 - d\).
    Therefore, because \(a\) is sortable in \((S_1, R'_1)\) and
    no \(d\)-sandwich blocks \(a\) in \(X\),
    the color \(a\) is also sortable in \((S_1, R_1)\).
    Recall that no color that appears after \(ax\) is in \(\theta_a\) and \(x\) is not in \(\theta_a\). 
    This, with the fact that \(axda \subseteq SdR \subseteq X\), implies that
    \(xda \subseteq R_1\).
    Therefore, sorting \(a\) in the state \((S_1, R_1)\) pushes \(xd\) onto the stack.
    Let \(y\) be the color that immediately precedes \(d\) in \(R_1\) (so we could have \(y = x\)).
    Note that since \(ax \subseteq Y\) and either \(x = y\) or \(xy \subseteq Y\), in all cases,
    we have
    \begin{equation}\label{ay}
      ay \subseteq Y.
    \end{equation}

    Recall that, by the previous arguments, 
    \(d\) is pushed on the stack when we attempt to apply \(\theta\) to \(X\), 
    so \(\theta\) cannot be a sorting sequence for \(X\) as otherwise
    we would have \((d, \emptyset) = \theta(X)\), a contradiction to the fact that
    the color \(d\) never appears as the top color on the stack.
    Therefore, there must exist a color in \(\theta\) that causes the stack to 
    be unsorted when \(\theta\) is applied to \(X\).
    Let \(z\) be the first such color and let 
    \(\theta_z\) be the list of initial colors in \(\theta\) up to but not including \(z\).
    That is, \(\theta_z\) is the longest initial sequence of \(\theta\) 
    that is a sorting sequence for \(X\).
    Let \((S_2, R_2) := \theta_z(X)\).  
    Note that by our previous arguments, \(a\) must appear in \(\theta_z\)
    and \(a\) is the color in \(W\) that is sorted first by \(\theta_z\)
    and \(d \notin W\).
    By Lemma~\ref{lem:interlinking}, we have 
    \begin{equation}\label{W}
      W \supseteq I(a,t_0, \dotsc, t_n)R_2 = t_0aI(t_0, \dotsc, t_n)R_2,
    \end{equation} 
    where \(t_0, \dotsc, t_n\) is the sequence of colors that appear above \(d\) on the stack
    after \(a\) is sorted.

    Note that since \(y\) is immediately below \(d\) on the stack after the sequence
    \(\theta_a\) is applied to \(X\) and \(d\) is never removed from the stack while 
    \(\theta_z\) is applied to \(X\), 
    the color \(y\) is immediately below \(d\) on the stack \(S_2\). 
    Let \(S_2' := S_2 - d\) and note that \((S_2', R_2) = \theta_z(X - d)\).
    By the selection of \(z\), 
    the color \(z\) is unsortable in \((S_2, R_2)\), but \(z\) is sortable in \((S'_2, R_2)\).
    Since \(d \notin S'_2R_2 \subseteq X - d\), we have \(z \neq d\) and \(dz \nsubseteq R_2\),
    so Lemma~\ref{lem:modsorting} implies that
    \(y\) is the top color, \(d\) is immediately below the color \(y\), and \(z\) is
    the color immediately below \(d\) on the stack \(S_2\)
    and that \(yz \subseteq R_2\).
    Note that this means that \(z = t_n\), since \(t_n\) is defined to be the last color that
    appears immediately above \(d\) on the stack when we sort \(X\) with \(\theta_z\).
    So, \(yt_n \subseteq R_2\).
    This, with \eqref{ay} and \eqref{W}, implies that
    \[ 
      X = YdW 
      \supseteq aydt_0aI(t_0, \dotsc, t_n)R_2 \supseteq 
      aydIt_0a(t_0, \dotsc, t_n)yt_n \cong I_{1,n}.
      \qedhere
    \]
\end{proof}

\medskip

The following three lemmas are needed in the proof of our
main theorem to show that a given unsortable sock ordering has a subpattern that is in \( \Gamma \).  Unsurprisingly, these proofs have a number of cases and are somewhat technical, so 
we defer their proofs until Section~\ref{sec:deferedproofs}.
\begin{restatable}{lemma}{greedytriv}\label{lem:greedytriv}
If \(X\) is trivially unsortable, then a subpattern of \(X\) is in \(\Gamma\).
\end{restatable}

\begin{restatable}{lemma}{greedythree}\label{lem:greedy3}
    Let \(X\) be a sock ordering and \(n \ge 0\). Suppose that a nonempty sorting sequence produces \((S,R)\) when it is applied
    to \(X\) and let \(d\) be the last color sorted. Let \(a,b,t_0,\dots,t_n\) be distinct colors in \(X\) which are different from \(d\) and define \(t:=t_0\). If \(abt \subseteq S\) and either
\begin{itemize}
\item \(atb \subseteq R\) or \(abt \subseteq R\) and \(a\) is blocked by a \(d\)-sandwich in \(X\), or 
\item \(bI(t_0, \dotsc, t_n)at_n \subseteq R\) and \(b\) is blocked by a \(d\)-sandwich in \(X\), 
\end{itemize}
then \(X\) is unsortable.
   If, in addition, every proper subordering of \(X\) is sortable, then a subpattern of \( X \) is in \(\Gamma\).
\end{restatable}

\begin{restatable}{lemma}{greedytwo}\label{lem:greedy2}
    Let \(X\) be a sock ordering and
    suppose that a nonempty sorting sequence produces \((S,R)\) when it is applied
    to \(X\) where \(d\) is the last color sorted by the sequence.
    Further suppose that the following two conditions hold:
    \begin{itemize}
        \item For every \(x \in S\), the state \( (x, R) \) is not terminal.
        \item There exist distinct colors \(a\) and \(t\) in \(S\) such that
        \( (at, R) \) is terminal. 
   \end{itemize}
   Then there exists \(P \subseteq R\) such that \(a\) appears in \(P\)
   and \(a\) is the only
   good sortable color in \(atP\).
   Furthermore, if
   \(a\) is blocked by a \(d\)-sandwich in \(X\),
   then \(X\) is unsortable.
   If, in addition, every proper subordering of \(X\) is sortable,
   then a subpattern of \( X \) is in \(\Gamma\).
\end{restatable}

The following lemma shows that if we reach the state
\((S,R)\) that is terminal and \(S\) has at least three colors, 
then, in a certain sense, we only
need to focus on at most three colors in \(S\).
Furthermore, two of the three colors can be the top two colors on \(S\).
\begin{lemma}\label{lem:greedy4}
    Let \(S\) and \(R\) be sock orderings where \(S\) is sorted and
    at least three distinct colors appear in \(S\).
    Suppose \((S, R)\) is terminal.
    If \(t\) is the top color and \(b\) is the color immediately below \(t\) on \(S\), 
    then there exists a color \(a \in S\) distinct from \(b\) and \(t\)
    such that \((abt, S)\) is terminal.
\end{lemma}
\begin{proof}
    Suppose \((S, R)\) is a counterexample.
    We can assume without loss of generality that no colors are repeated in \(S\).
    So we can  write \(S = x_1x_2 \dotsm x_mbt\) where
    \(x_1, \dotsc, x_m\), \(b\), and \(t\) are distinct colors.
    Let \(p\) be as small as possible such that there exists
    \(y_1, \dotsc, y_p\), a subsequence of \(x_1, \dotsc, x_m\), 
    such that \((y_1y_2 \dotsm y_pbt, R)\) is terminal. 
    Since \((S,R)\) is a counterexample, we have \(p \ge 2\).
    If \(p = 2\), let \(W = b\) and otherwise let \(W = y_3\dotsm y_pb\).
    By the minimality of \(p\), for \(i \in \{1,2\}\), there exists some \(z_i\) that is sortable in 
    \((y_iWt, R)\).

    For \(i \in \{1,2\}\),
    since \(z_i\) is sortable in \((y_iWt, R)\),
    we have \(z_i \notin y_iW\).
    In particular, this implies that \(z_i \neq y_i\) and \(z_i \neq b\).
    Furthermore, since \(z_i\) is sortable in \((y_iWt, R)\), there is no \(y_iby_i\) sandwich blocking \(z_i\) in \(y_iWtR\), 
    so \(y_iz_i \nsubseteq R\).

    We will now show that \(y_1\) is not sortable in \((y_2Wt, R)\).
    Recall that \((y_1y_2Wt, R) = (y_1y_2 \dotsm y_pbt, R)\) is terminal,
    so \(z_1\) is not sortable in \((y_1y_2Wt, R)\).
    Therefore, 
    we can apply Lemma~\ref{lem:modsorting} to deduce that \(z_1\) is unsortable in \((y_2Wt, R)\),
    because \(z_1 \neq y_1\), \(y_1z_1 \nsubseteq R\), and \(y_2\) is not the second color on the
    stack \(y_1y_2Wt\).
    Since \(z_1\) is unsortable in \((y_2Wt, R)\), \(y_1z_1 \nsubseteq R\), and \(z_1 \notin W\), 
    either \(z_1 = y_2\) or 
    there is a sandwich that blocks \(z_1\) in \(y_2WtR\) in which \(y_1\) does not appear.
    If \(z_1 = y_2\), then, because \(z_1 = y_2\) is sortable in \((y_1Wt, R)\), 
    we have \(y_2 \in R\), so 
    \(y_1y_2 = y_1z_1 \nsubseteq R\) implies that either \(y_1 \notin R\) or
    that \(y_1\) is blocked by a sandwich of the form \(y_2by_2\) in \(y_2WtR\).
    So, \(y_1\) is not sortable in \((y_2Wt, R)\).
    If there is a sandwich that blocks \(z_1\) in \(y_2WtR\) in which \(y_1\) does not appear,
    then \(y_1z_1 \nsubseteq R\) implies that \(y_1\) is not sortable in \((y_2Wt, R)\).
    So, in all cases, \(y_1\) is not sortable in \((y_2Wt, R)\).

    Since \(y_1\) is not sortable in \((y_2Wt, R)\), we have \(y_1 \neq z_2\).
    This, with the fact that \(z_2\) is not sortable in \((y_1y_2Wt, R)\) but 
    is sortable \((y_2Wt, R)\) together with Lemma~\ref{lem:modsorting},
    yields \(y_1z_2 \subseteq R\).
    Since \(y_1z_2 \subseteq R\) and \(y_2z_2 \nsubseteq R\), 
    \(y_2\) is not sortable in \((y_1Wt, R)\), so \(y_2 \neq z_1\).  
    Using this fact with Lemma~\ref{lem:modsorting}, allows us to conclude that 
    \(y_2z_1 \subseteq R\) because 
    \(z_1\) is not sortable in \((y_1y_2Wt, R)\) but is sortable \((y_1Wt, R)\). 
    Since \(y_1z_2 \subseteq R\), \(y_2z_2 \nsubseteq R\), and \(y_2z_1 \subseteq R\), we have \(y_1z_1 \subseteq R\).
    But we have already established that \(y_1z_1 \nsubseteq R\).
\end{proof}

Our last lemma is simple to prove despite its somewhat technical statement.
Informally speaking, we use the lemma in the following situation.
Suppose there are two distinct colors, say \(x\) and \(y\), and the final \(x\)
precedes the final \(y\).
We would expect the greedy algorithm to sort \(x\) before it sorts \(y\), but sometimes
this might not occur.
One reason \(y\) might be sorted before \(x\) is that, when \(y\) is sorted, \(x\) is blocked by a sandwich.
Note that, when this occurs, the sandwich blocking \(x\) must be a \(y\)-sandwich
because the greedy algorithm sorts \(y\) and the final \(y\) appears after the final \(x\).
The other possibility is that \(x\) cannot be sorted because it already appears
below the top color on the stack.
This lemma is used when we are in this second case.
\begin{lemma}\label{lem:secondinstack}
    Let \(S\) and \(R\) be sock orderings where \(S\) is sorted 
    and let \(x\) and \(y\) be distinct colors in \(R\) such that
    the final \(x\) in \(R\) appears before the 
    final \(y\) in \(R\). 
    Suppose that, with respect to \((S,R)\), 
    the color \(y\) is sortable and the color \(x\) is unsortable.
    If no \(y\)-sandwich blocks \(x\) in \(SR\),
    then \(y\) is the top color and \(x\) is the color immediately below \(y\) on \(S\).
    Furthermore, no color except possibly \(y\) precedes \(x\) in \(R\).
\end{lemma}
\begin{proof}
   Because \(y\) is sortable in \(SR\), 
   every sandwich in \(SR\) that appears before the final \(y\) in \(R\) must
   be a \(y\)-sandwich.
   Therefore, 
   because the final \(x\) in \(R\) precedes the final \(y\) in \(R\)
   and \(x\) is not blocked by a \(y\)-sandwich in \(SR\),
   the color \(x\) is not blocked by a sandwich in \(SR\).
   The fact that \(x\) is unsortable then implies that \(x\) 
   is below the top color on the stack \(S\).
   Since \(y\) is sortable in \(SR\), 
   there is no \(x\)-sandwich that blocks \(y\) in \(SR\).
   Therefore, \(y\) is the top color and \(x\) is the color immediately below \(y\) on \(S\).
   Furthermore, no color except possibly \(y\) can precede \(x\) in \(R\). 
\end{proof}

\section{Proof of Theorem~\ref{thm:main}}\label{sec:main_proof}

The proof is by induction on the number of socks in \(X\). 
The base case is vacuously true. 
So assume that \(X\) is an unsortable sock ordering.
With the induction hypothesis, 
we can assume that there does not exist a proper subordering of \(X\) that is unsortable.
Assume for a contradiction that \(X\) avoids \(\Gamma\).

\begin{claim}\label{clm:dsandwich}
  Let \(\theta\) be a nonempty sorting sequence for \(X\), let \((S, R) := \theta(X)\),
  and let \(d\) be the last color in \(\theta\).
  If there are distinct colors \(a\) and \(d\) and substrings \(S'\subseteq S \subseteq X\)  
  and \(R' \subseteq R \subseteq X\) such that 
\begin{itemize}
  \item \(a\) is not the top color of \(S'\),
  \item \(a \in R'\), and
  \item \(a\) is the only good sortable color in \(S'R'\),
\end{itemize}
then \(X\) contains a \(d\)-sandwich that blocks \(a\). 
\end{claim}
\begin{proof}
  This claim is simply a restatement of Lemma~\ref{lem:nodsandwich}.
  To see this, first note that the hypothesis of the claim is essentially 
  the same as the hypothesis of Lemma~\ref{lem:nodsandwich}.
  Furthermore, Lemma~\ref{lem:nodsandwich}\ref{II} cannot hold because we have
  assumed that every proper subordering is sortable;
  Lemma~\ref{lem:nodsandwich}\ref{III} cannot hold because \(X\) is unsortable;
  and Lemma~\ref{lem:nodsandwich}\ref{IV} cannot hold because we are assuming 
  that \(X\) avoids \(\Gamma\).
  Therefore, Lemma~\ref{lem:nodsandwich}\ref{I} must hold. 
  That is, there is a \(d\)-sandwich that blocks \(a\) in \(X\). 
\end{proof}

Let \((S, R)\) be the output of the greedy algorithm (as described in Section~\ref{sec:greedy})
applied to \(X\).
This implies that \((S, R)\) is terminal.
Since \(X\) is unsortable, we have \(R \neq \emptyset\).

\medskip

\noindent \textbf{Case 1:  \(S = \emptyset\) or there exists \(x \in S\) such that \((x, R)\) is terminal.}
If \(S = \emptyset\), then \(R\) is trivially unsortable. 
Similarly, if there exists \(x \in S\) such that \((x, R)\) is terminal then \(xR\) is trivially unsortable.
In either case, Lemma~\ref{lem:greedytriv} implies that a subpattern of \(X\) is in \(\Gamma\).
\medskip

From now on, we will assume that we are not in Case 1.  
That is, for every \(x \in S\), some color in \(xR\) is sortable.
Since \((S, R)\) is terminal, this implies that 
there are at least two colors on \(S\).
Define \(t\) to be the top color in \(S\), and let \(b\) be the color in \(S\) appearing immediately under \(t\).
We can also now let \(d\) be the last color sorted by the greedy algorithm.

\medskip
 
\noindent \textbf{Case 2: \( (bt, R)\) is terminal.}
By Lemma~\ref{lem:greedy2}, 
there exists \(P \subseteq R\) such
that \(b\) appears in \(P\) and 
such that \(b\) is the only good sortable color in \(btP\).
This means that the hypothesis of Claim~\ref{clm:dsandwich} is satisfied with
\(b\), \(d\), \(bt\), and \(P\) playing the roles
of
\(a\), \(d\), \(S'\), and \(R'\), respectively.
Therefore, \(b\) is blocked by a \(d\)-sandwich in \(X\).
Note that this is the hypothesis of the second statement in Lemma~\ref{lem:greedy2}.
The final statement of Lemma~\ref{lem:greedy2} then implies that a subpattern of \(X\) is in \(\Gamma\),
a contradiction.

\medskip

\noindent \textbf{Case 3: \( (bt, R)\) is not terminal.}
With Lemma~\ref{lem:greedy4}, we can define \(a\) to be the lowest color on the stack \(S\)
such that \((abt, R)\) is terminal.
\begin{claim}\label{clm:asinR}
    \(at\subseteq R\).
\end{claim}
\begin{proof}
    Suppose, by way of contradiction, that \(at\nsubseteq R\).
    Then, because \(t\) is unsortable in \((abt, R)\) and \(t \neq a\) and \(a\) is not the second color on the stack \(abt\), Lemma~\ref{lem:modsorting} implies
    that \(t\) is unsortable in \((bt, R)\).
    By the case, \((bt, R)\) is not terminal, so since both \(b\) and \(t\) are not sortable in 
    \((bt, R)\), there must exist a color \(y \in R\) that is not \(b\) or \(t\) which is sortable in \((bt, R)\).
    If \(y = a\), then, because \(at \nsubseteq R\), the color \(y = a\)
    must follow the last occurrence of \(t\) in \(btR\).
    If \(y \neq a\), then, 
    since \(y\) is sortable in \((bt, R)\) and unsortable in \((abt, R)\), 
    Lemma~\ref{lem:modsorting} implies that \(ay \subseteq R\).
    Because \(at \nsubseteq R\), this implies that \(y\) follows the last occurrence of
    \(t\) in \(btR\).
    So, in both cases, \(y\) follows the last occurrence of \(t\) in \(btR\).
    This with Lemma~\ref{lem:firsttwo} and the fact that \(t\) is the second color
    in \(btR\) and \(t\) is unsortable in \(btR\) 
    give us that \(y\) is unsortable in \(btR\), a contradiction.
\end{proof}

\medskip

\noindent \textbf{Case 3.1: There exists a color \(z \in S\) distinct from \(a\) and \(t\) 
such that  \(azt \subseteq S\) and \(z \in R\).}
By Claim~\ref{clm:asinR}, we can assume that either 
\(atz\), \(azt\), or \(zat = zI(t)at\) is in \(R\).
First suppose that either \(atz\) or \(azt\) is \(R\).
Because \(a\) is the only good sortable color in both \(aztatz\)
and \(aztazt\), 
Claim~\ref{clm:dsandwich} implies that \(a\) is blocked by a \(d\)-sandwich.
Since every proper subordering of \(X\) is unsortable,
the last statement of Lemma~\ref{lem:greedy3} implies that \(X\) does not avoid \(\Gamma\),
a contradiction.

The other case is when \(zat = zI(t)at\) is in \(R\) and the argument is similar.
Since \(z\) is the only good sortable color in \(aztzat\),
Claim~\ref{clm:dsandwich} implies that \(z\) is blocked by a \(d\)-sandwich.
So, like the previous case, Lemma~\ref{lem:greedy3} implies that \(X\) does not avoid \(\Gamma\),
a contradiction.
\medskip

\noindent \textbf{Case 3.2: For every \(z \in S\) distinct from \(a\) and \(t\) 
such that  \(azt \subseteq S\), we have \(z \notin R\).}
Recall that \(abt \subseteq S\).
This and the case imply that \(b \notin R\).
Let \(A\) and \(C\) be such that \(X = AbC\) and \(b \notin C\).  Because
\(b \in S\) and \(b \notin R\), there exists a color in \(C\) that is sorted by the greedy algorithm.
Let \(c \in C\) be the color in \(C\) that is sorted first by the greedy algorithm 
and let \((S', A'bC)\) be the state immediately before \(c\) is sorted.
Note that we have 
\begin{equation}\label{eq:S'A'subsetA}
S'A' \subseteq A.
\end{equation}
Recall that the color \(b\) is never sorted by the greedy algorithm and 
note that, after \(c\) is sorted, the color \(b\) only appears on the stack.
Therefore, because \(ab \subseteq S\), we have
\begin{equation}\label{eq:aonstack}
a \in S'A'
\end{equation}
Furthermore, since \(b\) is not sorted by the greedy algorithm, 
the color \(b\) can never appear as the top color on the stack after \(c\) is sorted 
by the greedy algorithm.
For \(n \ge 0\), let \(t_0, \dotsc, t_n\) (in that order from bottom-to-top)
be the \(n + 1\) distinct colors that appear immediately above \(b\) after \(c\) 
is sorted by the greedy algorithm.
Note that then \(t = t_n\), so Claim~\ref{clm:asinR} implies that \(at_n = at \subseteq R\).
With Lemma~\ref{lem:interlinking}, we have 
\begin{equation}\label{eq:Cstructure}
C \supseteq I(c,t_0, \dotsc, t_n)R \supseteq  I(c, t_0, \dotsc, t_n)at_n.
\end{equation}

Note that if \(ca \subseteq A\), then, with \eqref{eq:Cstructure}, we have
\[
   X = AbC \supseteq cabI(c, t_0, \dotsc, t_n)at_n = cabt_0c I(t_0, \dotsc, t_n)at_n \cong I_{1,n}. 
\]
So, for the remainder of the proof, we will assume 
\begin{equation}\label{eq:canotinA}
ca \nsubseteq A.
\end{equation}

\begin{claim}\label{clm:azc}
  Suppose that there is a color \(z \in A'bC\) distinct from \(a\) and \(c\) 
  such that \(z\) is not blocked by a \(c\)-sandwich
  and the final \(z\) appears before the final \(c\) in \(A'bC\).
  Then \(z\) is the color immediately beneath the top color on the stack \(S'\) and 
  \(azc \subseteq S'\) and \(zI(c,t_0, \dotsc, t_n)at_n \subseteq A'bC\).
\end{claim}
\begin{proof}
  Recall that, in the state \((S', A'bC)\), the greedy algorithm sorted the color \(c\). 
  Since the final \(z\) precedes the final \(c\) in \(S'A'bC\), 
  this implies that, with respect to the state \((S', A'bC)\), 
  the color \(c\) is sortable and the color \(z\) is not sortable.
  By Lemma~\ref{lem:secondinstack} (with \(S'\), \(A'bC\), \(c\) and \(z\) playing the roles
  of \(S\), \(R\), \(y\), and \(x\), respectively),
  \(c\) is the top color and \(z\) is the color immediately below \(c\) on the stack \(S'\)
  and no color except possibly \(c\) precedes \(z\) in \(A'bC\).
  So, with \eqref{eq:Cstructure}, we have 
  \[
    A'bC \supseteq zC \supseteq zI(c,t_0, \dotsc, t_n)at_n.
  \]

  Recall that \(a \in S'A' \subseteq A\) by \eqref{eq:aonstack} and \eqref{eq:S'A'subsetA}
  and note that \eqref{eq:canotinA} implies that the final \(a\) in \(A\) precedes the first \(c\)
  in \(A\).  Therefore, the assumption that \(c\) is on the stack \(S'\) implies
  that \(a \subseteq S'\) as well.
  Furthermore, \(c\) is the top color on \(S'\) and \(z\) is the color immediately below \(c\) on \(S'\), 
  we have \(azc \subseteq S'\). 
\end{proof}

\begin{claim}\label{clm:Sprime}
There are at least three colors in \(S'\)
and when we define \(y\) to be the color immediately beneath the top color in the stack \(S'\),
we have \(y \neq a\), \(y \neq c\), \(ayc \subseteq S'\), and  \(yI(c,t_0, \dotsc, t_n)at_n \subseteq A'bC\).
\end{claim}
\begin{proof}
  If \(b\) is not blocked by a \(c\)-sandwich, then
  because \(b\) is distinct from \(a\) and \(c\) and the final \(b\) occurs 
  before the final \(c\) in \(S'A'bC\),
  the hypothesis of Claim~\ref{clm:azc} is satisfied with \(b\) playing the role of \(z\). 
  This implies that \(y = b\), and that the desired conclusion holds.

So we can assume that \(b\) is blocked by a \(c\)-sandwich in \(S'A'bC\).
Let \(\hat{y}\) be the other color in the first \(c\)-sandwich that blocks \(b\) in \(S'A'bC\). 
By definition,
\begin{equation}\label{eq:hatynotc}
  \hat{y} \neq c.
\end{equation}
We will show that \(\hat{y}=y\) and that the desired conclusion holds.
To see this,  first note that we must have \(\hat{y}c \subseteq S'A' \subseteq A\)
and \(c\hat{y} \subseteq S'A' \subseteq A\). Furthermore, 
\eqref{eq:canotinA} and \(c\hat{y} \subseteq S'A' \subseteq A\) imply that
\begin{equation}\label{eq:hatynota}
\hat{y} \neq a 
\end{equation}
and with \eqref{eq:aonstack}
\begin{equation}\label{eq:ahaty}
a\hat{y} \subseteq S'A'.
\end{equation}

If the final \(\hat{y}\) follows the final \(c\) in \(S'A'bC\), then,
after \(c\) is sorted,
the color \(\hat{y}\) must appear off the stack
and, with \eqref{eq:ahaty}, we will also have \(a\hat{y}b\) on the stack.
Then, the fact that \(b\) is never sorted by the greedy algorithm implies
that \(\hat{y}\) is never sorted by the greedy algorithm.
Furthermore, since we cannot have a \(\hat{y}b\hat{y}\)-sandwich on the stack, 
we must have \(a\hat{y}bt_n \subseteq S\) and \(\hat{y} \in R\). 
But this contradicts the fact that we are in Case 3.2.

Therefore, the final \(\hat{y}\) precedes the final \(c\) in \(S'A'bC\).
Recall that \(\hat{y}\) is the color other than \(c\) that is 
in the first \(c\)-sandwich that blocks \(b\) in \(S'A'bC\).
Since \(b \notin C\), we have \(\hat{y}b \subseteq S'A'b\).
Because \(c\) is sortable in \(S'A'bC\), 
there is no sandwich consisting of \(b\) and \(\hat{y}\) 
which blocks \(c\) in \(S'A'bC\)
and both the final \(b\) and the final \(\hat{y}\) in \(S'A'bC\) precede the final
\(c\) in \(S'A'bC\),
this implies that \(b \hat{y} \nsubseteq S'A'bC\).
In particular, we have \(\hat{y} \notin C\).
We claim that the color \(\hat{y}\) is not blocked by a \(c\)-sandwich in \(S'A'bC\).
For a contradiction, assume that \(\bar{y}\) is the color other than \(c\) in such a sandwich.
Since \(\hat{y} \notin C\), this sandwich is in \(S'A'b\),
and \(\bar{y}\hat{y} \subseteq S'A'b\). 
Since \(b \hat{y} \nsubseteq S'A'bC\), this implies that \(\bar{y} \neq b\).
Therefore, there is a sandwich consisting of \(\bar{y} \neq b\) and \(c\) 
in \(S'A'b\), so this sandwich blocks \(b\) as well as \(\hat{y}\).
Furthermore, since there is no sandwich consisting of \(\bar{y}\) and \(\hat{y}\)
that blocks \(c\) in \(S'AbC\), we have \(\hat{y} \bar{y} \nsubseteq S'Ab\).
But this contradicts the fact that \(\hat{y}\) is in the first \(c\)-sandwich in \(S'A'bC\)
that blocks \(b\).

So we have established that the color \(\hat{y}\) is not blocked by a \(c\)-sandwich  
and the final \(\hat{y}\) precedes the final \(c\) in \(S'A'bC\).
Also, \eqref{eq:hatynotc} and \eqref{eq:hatynota} imply
\(\hat{y}\) is distinct from \(a\) and \(c\). 
So, the hypothesis of Claim~\ref{clm:azc} is satisfied with  
\(\hat{y}\) playing the role of \(z\).
Therefore, \(y = \hat{y}\) and the desired conclusion holds.
\end{proof}

Define \(d'\) to be the color sorted by the greedy algorithm immediately
before reaching the state \((S', A'bC)\).
By Claim~\ref{clm:Sprime}, we have \(ayc \subseteq S'\)
and \(yI(c, t_0, \dotsc, t_n)at_n \subseteq A'bC\).
Note that \(y\) and \(c\) are the only sortable colors in \(aycyI(c, t_0, \dotsc, t_n)at_n\),
because every other color is blocked by a \(ycy\) sandwich.
Furthermore, after sorting \(c\) we have \((ayyt_0, I(t_0, \dotsc, t_n)at_n)\) as a substate which Observation~\ref{obs:inf_unsort} implies is unsortable. 
Therefore, with Observation~\ref{obs:inf_unsort}, \(y\) is the only good sortable color in \(aycyI(c, t_0, \dotsc, t_n)at_n\).
This means that the hypothesis of Claim~\ref{clm:dsandwich} is satisfied with
\(y\), \(d'\), \(ayc\), and \(yI(c, t_0, \dotsc, t_n)a t_n\) playing the roles
of \(a\), \(d\), \(S'\) and \(R'\), respectively.
Therefore, we can assume that \(y\) is blocked by a \(d'\)-sandwich in \(X\).
Then, the last statement of Lemma~\ref{lem:greedy3} implies that a subpattern of \(X\) is in \(\Gamma\).

\section{Proofs of Lemmas~\ref{lem:greedytriv}, \ref{lem:greedy3}, and \ref{lem:greedy2}}\label{sec:deferedproofs}

\greedytriv*
\begin{proof}
    Assume, for a contradiction, that \(X\) is a  minimal counterexample.
    Let \(aba\) be the first sandwich in \(X\).
    Recall that since \(X\) is trivially unsortable, every color is blocked
    by a sandwich.  That is, for all \(z\in X\), there exist distinct colors \(x\) and \(y\) in \(X\) which are distinct from \(z\) such that \(xyxz\subseteq X\). 
    This implies that every color must follow the first sandwich \(aba\).
    Let \(z\) be the first of the colors \(a\) and \(b\) that terminates in \(X\). 
    Let \(xyx\) be the first sandwich in \(X\) that blocks \(z\).

    \textbf{Case 1:} \(x = a\).
    
    Then \(xyx = aya\) blocks \(z\), so \(z = b\).
    Suppose the \(y\) in the sandwich \(aya\) follows the first sandwich \(aba\), then 
    \(X \supseteq abayaba \cong abacaba \cong \mathcal{T}_1\).
    Otherwise, because \(aba\) (and not \(aya\)) is the first sandwich, we have \(X \supseteq abya\).
    This \(abya\) subordering is followed by \(ba\) and the color \(y\) in \(X\).
    Therefore, either 
    \(X \supseteq abyabya \cong abcabca \cong \mathcal{T}_3\), \(X \supseteq abyayba \cong abcacba \cong \mathcal{T}_4\), or 
    \(X \supseteq abyabay \cong abcabac \cong \mathcal{T}_5\). 

    \textbf{Case 2:} \(x = b\).
    
    Then \(xyx = byb\) blocks \(z\), so \(z = a\).
    Suppose the \(y\) in the sandwich \(byb\) follows the first sandwich \(aba\). Then 
    \(X \supseteq abaybab \cong abacbab \cong \mathcal{T}_2\).
    Otherwise, because \(aba\) is the first sandwich (and not \(aya\)), we have \(X \supseteq abya\).
    Since \(byb\) blocks \(z = a\) and the final \(b\) follows
    the final \(z = a\),  the subordering \(bab\) must follow the subordering \(abya\).
    Furthermore, since \(y\) must follow the first sandwich, 
    the color \(y\) must follow this \(abya\).
    So, as in the previous case,  \(abya\) is followed by \(ba\) and the color \(y\).
    Therefore, as in the previous case, either \(\mathcal{T}_3\), \(\mathcal{T}_4\), or \(\mathcal{T}_5\) is a subpattern of \(X\).

    \textbf{Case 3:} \(x \neq a\) and \(x \neq b\).

    \textbf{Case 3.1:} The first \(x\) appears before the first sandwich \(aba\).
    
    Then, because \(aba\) is the first
    sandwich to appear, the color \(x\) must also follow \(aba\).
    Therefore, \(X \subseteq xabaxab \cong abcbacb \cong \mathcal{T}_6\) 
    or \(X \subseteq xabaxba \cong abcbabc \cong \mathcal{T}_7\).

    \textbf{Case 3.2:} The first \(x\) appears between the first and last \(a\)
    in the first sandwich \(aba\).

    By the case, an \(axa\) sandwich precedes the first \(xyx\) sandwich.
    Since the first sandwich that blocks \(z\) is \(xyx\), we must have \(z = a\).
    Note that, because \(xyx\) blocks \(z = a\), we also have \(y \neq a\).
    
    The color \(x\) must follow the \(ab\) in the first sandwich \(aba\), as
    otherwise an \(axa\) sandwich would proceed the first \(aba\) sandwich.
    Since \(xyx\) is not the first sandwich in \(X\) and \(y \neq a\), 
    either \(X \supseteq abxayx\) or \(X \supseteq abxyax\).
    We can assume that both \(aba\) and \(xyx\) in this subordering 
    are the first such subordering to appear in \(X\).
    If \(y = b\), then  the subordering is
    \(abxabx\) or \(abxbax\). But then \(bxb\) precedes \(xbx = xyx\),
    so \(xbx  = xyx\) is not the first sandwich which blocks \(z = a\).
    Therefore, we can assume  \(y \neq b\). 
    Therefore, the colors \(a\) and \(b\) are distinct
    from the colors \(x\) and \(y\), so
    either \(X \supseteq abxayxab \cong abcadcab \cong \mathcal{T}_8\) or the subordering is \(abxyax\).
    So, we can assume the subordering is \(abxyax\).
    
    Recall that every color must follow the first \(aba\) sandwich, so every color must follow \(abxya\). 
    This with the fact that \(a\) is neither \(x\) nor \(y\)
    implies that every color that is not \(a\) is blocked by either
    \(axa\) or \(aya\).
    Furthermore, \(a\) is blocked by \(xyx\).
    Therefore, every color is blocked by one of the three sandwiches
    \(axa\), \(aya\), or \(xyx\).
    Since none of these three sandwiches includes the color \(b\), 
    if we remove the color \(b\) from \(X\) we have a trivially unsortable sock ordering
    which contradicts the minimality of \(X\).

    \textbf{Case 3.3:} The first \(x\) appears after the first sandwich \(aba\).
    
    So, we have \(abaxyx\) in \(X\) where \(aba\) and \(xyx\) in this subordering
    are the first such subordering to appear in \(X\).  
    If \(y = a\), then we must have \(z = b\), but then
    \(axy = axa\) appears before \(xyx\), a contradiction to the fact that
    \(xyx\) is the first sandwich that blocks \(z = b\).
    Similarly, if \(y = b\), then we must have \(z = a\), but then \(bxy = bxb\)
    blocks \(z=a\) and appears before \(xyx\).
    Therefore, the colors \(x\) and \(y\) are distinct from the colors \(a\) and \(b\).
    Hence, 
    \(X \supseteq abaxyxba \cong abacdcba \cong \mathcal{T}_9\) or \(X \supseteq abaxyxab \cong abacdcab \cong \mathcal{T}_{10}\).
\end{proof}

\greedythree*
\begin{proof}
First note that if a proper subordering of \(X\) is unsortable, then \(X\) is unsortable.
Therefore, to prove the lemma, we can assume that every proper subordering of \(X\) is sortable,
and prove that a subpattern of \(X\) is in \(\Gamma\), since this will imply that \(X\) is unsortable.
Assume for a contradiction that no subpattern of \(X\) is in \(\Gamma\).
Let 
\(T := (t_0, \dotsc, t_n) \) and 
\(P \in \{atb, abt, bI(T)at_n\}\) such that \(P \subseteq R\).
We may write \(X = {\mS}d{\mR}\) where \(S \subseteq {\mS}\) and \(d \notin {\mR}\) and \(R \subseteq {\mR}\).

Note that \(a\) is the only sortable color in 
\(abtatb\) and \(abtabt\).
We claim that \(b\) is the only good sortable color
in \(abtbI(T)at_n\).  To see this note that when \(n = 0\) the color \(b\) is the only
sortable color in  \(abtbI(T)at_n = abtbat\) and when \(n \ge 1\),
\(b\) and \(t\) are the only sortable colors in \(abtbI(T)at_n\), but
sorting \(t = t_0\) leaves \((abbt_1, I(t_1, \dotsc, t_n)at_n)\) which is unsortable by Observation~\ref{obs:inf_unsort}.
Define
\[
  (z, y) := 
  \begin{cases}
  (a,b) & \text{if \(P \in \{atb, abt\}\),} \\
  (b,a) & \text{if \(P = bI(T)at_n\)}.
  \end{cases}
\]

We have the following.
\begin{claim}\label{clm:zblock}
The color \(z\) is blocked by a \(d\)-sandwich in \(X\) and 
the color \(z\) is not blocked by a sandwich consisting only of colors in \(P\).
\end{claim}
\begin{proof}
The first conjunct is given by the hypothesis of the lemma.
For the second conjunct, recall that, by the above discussion, \(z\) is the only good sortable color in \(abtP \subseteq X - d\).
Therefore, because \(X - d\) is a proper subordering of \(X\) and is therefore sortable, 
\(z\) is not blocked by a sandwich consisting only of colors in \(P\).
\end{proof}

\begin{claim}\label{clm:3stackfinal}
One of the following holds.
  \begin{enumerate}[label=(\roman*)]
      \item\label{3sa} one of \(dabt\), \(adbt\), \(abdt\), 
        is a subordering of \({\mS}\) or \(abtde \subseteq {\mS}\) for some color \(e \notin P\); or
      \item\label{3sb} \(atab \subseteq {\mR}\) or \(tP \subseteq {\mR}\); or
      \item\label{3sc} for some color \(e \notin P\), 
      \(e \in {\mS}\) and \(eP \subseteq {\mR}\).
  \end{enumerate}
\end{claim}
\begin{proof}
Recall that by Claim~\ref{clm:zblock}, \(z\) is blocked by a \(d\)-sandwich in \(X\).
That is, for some color \(e\), we either have a sandwich of the form
\(ded\) which blocks \(z\) or a sandwich of the form \(ede\) which blocks \(z\).

\textbf{Case 1: A sandwich of the form \(ded\) blocks \(z\)}.
In this case, we will show that \ref{3sa} holds, so 
suppose that neither \(dabt\), \(adbt\), nor \(abdt\) is a subordering of \({\mS}\).
Then, because \(d \notin {\mR}\) implies \(de \subseteq {\mS}\), we have 
\(abtde \subseteq {\mS}\).
To show that \ref{3sa} holds, we then only need to show that \(e \notin P\).
Since \(ded\) blocks \(z\), we have \(e \neq z\).
We also have \(e \neq y\), since otherwise we would have a \(yty\) sandwich
in \({\mS}\) which would block \(z\) and this is forbidden by Claim~\ref{clm:zblock}.
Furthermore, \(e \neq t\), since we are assuming \(abdt \subseteq abtdt\) is not in \({\mS}\).
So, we are left to consider the case when \(n \ge 1\), \(bI(T)at_n\) is in \(R\),
and \(e = t_i\) for some \(1 \le i \le n\).
Then 
\[
X \supseteq abtdt_idbI(T)at_n \supseteq btdt_ibtt_i \cong dabtdat \cong I_{1,0},
\] 
a contradiction.

\textbf{Case 2: A sandwich of the form \(ede\) blocks \(z\)}.
Since \(d \notin {\mR}\), we must have \(e \in {\mS}\).
If \(ede \subseteq {\mS}\), then a sandwich of the form \(ded\) blocks \(z\), so we are
in the first case.
Therefore, we can assume that \(ez \subseteq {\mR}\).

By Claim~\ref{clm:zblock}, 
the color \(z\) is not blocked by a \(yty\) sandwich. 
This, with the fact that \(yt \subseteq abt \subseteq S \subseteq {\mS}\) and \(ez \subseteq {\mR}\),
yields \(e \neq y\).
Since \(ede\) blocks \(z\), we have \(e \neq z\).
Therefore, \(e \notin \{y, z\} = \{a, b\} \).

\textbf{Case 2.1: \(P \in \{atb, abt, bat = bI(t_0)at_0\} \)}.
Recall that Claim~\ref{clm:zblock} implies that a \(yty\) 
sandwich does not block \(z\) in \(X\).
Since \(yt \subseteq abt \subseteq S\), \(z\) cannot follow the \(y\) in \(P\).
This, with the fact that \(ez \subseteq {\mR}\), 
implies that we have \(eP \subseteq {\mR}\) when \(P\) begins with \(zy\).
So, we have \(eP \subseteq {\mR}\) when \(P \in \{ abt, bI(t)at = bat \} \).
Now suppose \(P = atb\) and \(eP \nsubseteq {\mR}\). 
Then, because \(ea = ez \subseteq {\mR}\)
and \(z=a\) cannot follow \(y = b\) in \(P\), 
we either have 
\(e = t\) and \(atab \subseteq {\mR}\) or
\(e \neq t\) and \(aetab \subseteq {\mR}\) or \(ateab \subseteq {\mR}\).
Note that, in all cases, we have \(atab \subseteq {\mR}\), so \ref{3sb} is satisfied.
Therefore, we can assume that  \(eP \subseteq {\mR}\).  
If \(e = t\), then \(tP \subseteq {\mR}\) and condition \ref{3sb} is satisfied.
If \(e \neq t\), then because we have already shown that \(e \notin \{a,b\}\) and we are in Case 2.1, 
we have \(e \notin P\), so condition \ref{3sc} is satisfied.

\textbf{Case 2.2: \(n \ge 1\) and \(P = bI(T)at_n = bt_1tI(t_1, \dotsc, t_n)at_n\).}

In this case, we have \(z = b\) and \(y = a\).  
The color \(b\) cannot follow \(t_1t\) in \({\mR}\) because Claim~\ref{clm:zblock} 
implies that \(b\) is not blocked by a \(tt_1t\) sandwich. 
Therefore, because \(P = bt_1tI(t_1,\dotsc,t_n)at_n\) and 
\(eb \subseteq {\mR}\), if we do not have \(eP \subseteq {\mR}\), then we must
have \(bt_1btI(t_1,\dotsc,t_n)at_n \subseteq {\mR}\). 
But then, 
\[
X \supseteq abtdR \supseteq abtdbt_1btI(t_1, \dotsc, t_n)at_n 
\supseteq btdt_1btt_1 \cong dabtdat \cong I_{1,0},
\]
a contradiction.

Therefore, we can assume that \({\mR} \supseteq eP = ebt_1tI(t_1, \dotsc, t_n)at_n\).
Note that if \(e = t\), then \ref{3sb} holds and
if \(e \notin P\), then \ref{3sc} holds.
Since we have already shown that \(e \notin \{a, b\}\), we can assume \(e = t_i\) where 
\(1 \le i \le n\).
Recall that \(t_i = e\in {\mS}\), so since Claim~\ref{clm:zblock} implies that 
no \(t_itt_i\) sandwich blocks \(b\) in \(X\), we must have
\(abtt_i \subseteq {\mS}\).
Therefore, 
\[
X \supseteq abtt_idt_ibt_1tI(t_1, \dotsc t_n)at_n \supseteq btdt_ibtt_i \cong
dabtdat \cong I_{1,0},
\]
a contradiction.
\end{proof}

Suppose Claim~\ref{clm:3stackfinal}\ref{3sa} holds.
Note that \(at \subseteq P\) or \(n \ge 1\) and \(I(T)at_n \subseteq P\). 
Therefore,  if \(dabt \subseteq {\mS}\), then 
either
\(X \supseteq dabtdat \cong I_{1,0}\) or
\(X \supseteq dabtdI(T)at_n \cong I_{1,n}\), a contradiction.
Otherwise, one of the orderings in the following table is a
subordering of \(X\).
\begin{center}
    \begin{tabular}{|c||c||c|}
    \hline

\(adbtdatb \cong \mathcal{C}_{2,4}\)  & 
\(adbtdabt \cong \mathcal{C}_{2,3}\) &
\(adbtdbI(T)at_n \cong I_{2,n}\) \\\hline

\(abdtdatb \cong \mathcal{C}_{2,5}\)  & 
\(abdtdabt \supseteq abdtabt \cong I_{1,0}\) &
\(abdtdbI(T)at_n \cong I_{3,n}\) \\\hline

\(abtdedatb \cong \mathcal{C}_{2,7}\)  & 
\(abtdedabt \cong \mathcal{C}_{2,6}\) &
\(abtdedbI(T)at_n \cong I_{5,n}\) \\\hline

\end{tabular}
\end{center}

If Claim~\ref{clm:3stackfinal}\ref{3sc} holds, then by swapping the labels of \(d\) and
\(e\), we can see that either \(I_{1,n} \subpattern X\) or one of the orderings in the table above is a subpattern of \(X\).

Finally, if Claim~\ref{clm:3stackfinal}\ref{3sb} holds, then 
one of the four orderings in the following table is a subordering of \(X\).
\begin{center}
    \begin{tabular}{|c||c||c||c|}
    \hline
\(abtdatab \cong \mathcal{C}_{2,1}\) &
\(abtdtatb \cong \mathcal{C}_{2,2}\) &
\(abtdtabt \supseteq abdtabt \cong I_{1,0}\) &
\(abtdtbI(T)at_n \cong I_{4,n}\) \\\hline
\end{tabular}
\end{center}
\end{proof}

\greedytwo*
\begin{proof}
As in the proof of Lemma~\ref{lem:greedy3}, 
to prove the lemma, we can assume that every proper subordering of \(X\) is sortable,
and prove that a subpattern of \(X\) is in \(\Gamma\).
  For a contradiction, assume that no subpattern of \(X\) is in \(\Gamma\).
We may write \(X = {\mS}d{\mR}\) where \(S \subseteq {\mS}\) and \(d \notin {\mR}\) and \(R \subseteq {\mR}\).
Since \((at, R)\) is terminal, there exists a sandwich that blocks \(t\) in
\(atR\). Therefore, there exists a color that is neither \(a\) nor \(t\) in \(R\).
Let \(r\) be the first color that is neither \(a\) nor \(t\) in \(R\).
Furthermore, since a color that is neither \(a\) nor \(t\) must precede \(t\) in \(R\), 
we have \(rt \subseteq R\).

\begin{claim}\label{clm:aratinR}
    \(ar\subseteq R\) and \(at \subseteq R\).
\end{claim}
\begin{proof}
    Let \(x \in \{r,t\}\) and suppose, for a contradiction, that \(ax \nsubseteq R\).
    Note that then \(x\) is one of the first two colors to appear in \(tR\). 
    Lemma~\ref{lem:modsorting} implies that \(x\) is unsortable in \((t, R)\),
    because \(x\) is unsortable in \((at, R)\) and \(x \neq a\) and \(ax \nsubseteq R\).
    Recall that by the hypothesis of the lemma, there exists \(y\) that is sortable in \((t, R)\).
    Since \(y\) is unsortable in \((at, R)\), Lemma~\ref{lem:modsorting} implies that either \(y = a\) or \(ay \subseteq R\).
    In either case, \(y\) must follow the last occurrence of \(x\) in \(R\) 
    because \(ax \nsubseteq R\).
    However, with Lemma~\ref{lem:firsttwo},
    the fact that \(y\) is sortable in \((t, R)\) while \(x\) is unsortable
    in \((t, R)\), and the fact that
    \(x\) is one of the first two colors to appear in \(tR\), imply that 
    \(y\) cannot follow the last occurrence of \(x\) in \(R\), a contradiction.
\end{proof}

\begin{claim}\label{clm:PinR}
There are distinct colors \(c\) and \(b\) that are distinct from \(a\), \(r\), and \(t\),
such that there exists \(P\in\{arat,rart,arbrt, arcbct\}\) such that \(P \subseteq R\).
   In particular, we have \(art \subseteq R\).
\end{claim}
\begin{proof}
First, suppose that \(rat \subseteq R\).
Since, by Claim~\ref{clm:aratinR}, we have \(ar \subseteq R\), one of \(arat\), \(rart\), or \(ratr\) appears
in \(R\). 
Since \(atdratr = atdraI(r)tr \cong I_{1,0}\),
we can assume the subordering \(ratr\) does not appear \(R\),
so we are done in this case.

Now, suppose  \(rat \nsubseteq R\).
Since \(r\) is the first color in \(R\) that is not \(a\) or \(t\),
\(rat \nsubseteq R\) implies that 
the first sandwich that blocks \(t\) in \(atR\) does not include the color \(a\).
So there exist distinct colors \(c\) and \(b\) that are both neither \(a\) nor \(t\), 
such that a sandwich of the form \(cbc\) blocks \(t\) in \(atR\).
Since \(r\) is the first color that is not \(a\) or \(t\) to appear in \(R\),
we have either \(r = c\) and 
\(rbrt \subseteq R\) or \(r \notin \{b, c\}\) and \(rcbct \subseteq R\).
By Claim~\ref{clm:aratinR}, we have \(at \subseteq R\).
This, with the fact that \(rat \nsubseteq R\), 
implies that either 
\(arbrt \subseteq R\) or \(arcbct \subseteq R\),
where the distinct colors \(b\) and \(c\) are distinct from \(a\), \(r\), and \(t\).
\end{proof}

\begin{claim}\label{clm:nosandwich}
   For distinct colors \(a\), \(r\), \(t\), \(b\), and \(c\) and 
   every \(P \in \{arat,rart,arbrt, arcbct\}\) the color \(a\) 
   is the only sortable color in \(atP\)
   and \(a\) is a good sortable color in \(atP\).
   Therefore, no sandwich consisting only of the colors that appear in \(atP\) 
   blocks \(a\) in \(X\).  
\end{claim}
\begin{proof}
    It is not hard to see that \(a\) is the only sortable color in \(atP \subseteq {\mS}R\).
    Furthermore, it is not hard to see that \(a\) is a good sortable color in \(atP \subseteq {\mS}R\).
    Recall that, by assumption, \(X - d\) is sortable.
    This means that for any sorting sequence that sorts \(X - d\),
    the color \(a\) must be sorted first among the colors in \(atP\).
    Therefore, there is no sandwich consisting only of the colors 
    that appear in \(atP\) that blocks \(a\) in \(X\).
\end{proof}

Note that Claims~\ref{clm:PinR} and \ref{clm:nosandwich} establish the first 
statement of the conclusion.

\begin{claim}\label{clm:2stackfinal}
Let \(P \in \{arat, rart,arbrt, arcbct\}\) such that \(P \subseteq R\) where
\(a\), \(r\), \(t\), \(b\), and \(c\) are distinct colors.
One of the following holds.
  \begin{enumerate}[label=(\roman*)]
      \item\label{2sa} 
        \(dat\) or \(adt\) is a subordering of \({\mS}\) or 
        \(atde \subseteq {\mS}\) for some color \(e \notin P\); or
      \item\label{2sb} 
        \(taxat\), \(atxat\), \(txaxt\), or \(taxyxt\) is a subordering of \({\mR}\) 
        where \(a\), \(t\), \(x\), and \(y\) are distinct colors; or
      \item\label{2sc} 
        \(eP \subseteq {\mR}\) for some color \(e \notin P\) 
        and \(eat\), \(aet\), or \(ate\) is a subordering of \({\mS}\).
  \end{enumerate}
\end{claim}
\begin{proof}
By assumption, there exists a \(d\)-sandwich
that blocks \(a\) in \(X = {\mS}d{\mR}\).
That is, for some color \(e\), we either have a sandwich of the form
\(ded\) which blocks \(a\) or a sandwich of the form \(ede\) which blocks \(a\).
Therefore, we have the following two cases.

\textbf{Case 1: A sandwich of the form \(ded\) blocks \(a\) in \(X\)}.
Note that since \(d \notin {\mR}\), we have \(de \subseteq {\mS}\).
So, if neither \(dat \subseteq {\mS}\) nor \(adt \subseteq {\mS}\), 
then \(atde \subseteq {\mS}\) and \(e \neq t\).
So to show that \(X\) satisfies \ref{2sa}, we 
only need to show that \(e \notin P\).
First note that \(e \neq a\), since \(ded\) blocks \(a\).
If \(e = r\), then, since Claim~\ref{clm:PinR} implies that \(art \subseteq P\),
we have 
\(X \supseteq {\mS}dP \supseteq atdrdart \cong abdtdatb \cong C_{2,5}\), a contradiction.
We get a similar contradiction if \(P = arbrt\) or \(P = arcbct\) and \(e = b\), since in
either case \(abt \subseteq P\), so 
\({\mS}dP \supseteq atdbdabt \cong abdtdatb \cong C_{2,5}\), a contradiction.
If \(P = arcbct\) and \(e = c\), then \(act \subseteq P\),
so \({\mS}dP \supseteq atdcdact \cong abdtdatb \cong C_{2,5}\), a contradiction.
With Claim~\ref{clm:PinR}, this completes the proof of this case.

\textbf{Case 2: A sandwich of the form \(ede\) blocks \(a\) in \(X\)}.
Since \(d \notin R\), we have \(e \in {\mS}\). 
Furthermore, if \(de \subseteq {\mS}\), then a sandwich of the form \(ded\) blocks
\(a\) in \(X\), and this was handled in Case 1.
Therefore, we can assume that \(ea \subseteq {\mR}\).

\textbf{Case 2.1: \(e = t\)}.
In this case, \(ta = ea \subseteq {\mR}\) and, by Claim~\ref{clm:nosandwich} a sandwich
of the form \(trt\) does not block \(a\) in \(X = {\mS}d{\mR}\). 
The fact that \(at \subseteq {\mS}\) then implies that  \(rta \nsubseteq {\mR}\).
Therefore, there exists a \(t\) in \({\mR}\) that is before both an \(a\) 
in \({\mR}\) and the first \(r\) in \({\mR}\).
So,
if \(P = arat\), then we have either \(tarat = eP\) or \(atrat\) in \({\mR}\), so \ref{2sb} holds.
Furthermore, if \(P = rart\), we have \(trart = eP\) in \({\mR}\), so \ref{2sb} holds.
If \(P = arbrt\), then one option is \(tarbrt = eP \subseteq {\mR}\), and \ref{2sb} holds in this case.
By Claim~\ref{clm:nosandwich}, there is no \(rbr\) sandwich that blocks \(a\),
so the only other possibilities when \(P = arbrt\) are 
\(atrabrt \subseteq {\mR}\) or 
\(atrbart \subseteq {\mR}\), and, in either case, 
we have \(trart \subseteq {\mR}\), so \ref{2sb} holds.
Finally, assume \(P = arcbct\). If \( tarcbct = eP \subseteq {\mR}\), 
then \(tacbct \subseteq {\mR}\), so \ref{2sb} holds.
Otherwise, because there is no \(cbc\) sandwich that blocks \(a\), 
at least one of \(atracbct \supseteq tacbct \), \(atrcabct \supseteq tcact\), or \(atrcbact \supseteq tcact\) 
is a subordering of \({\mR}\), so \ref{2sb} holds.

\textbf{Case 2.2: \(e = r\)}.
In this case, \(r \in {\mS}\) and \(ra = ea \subseteq {\mR}\).
Recall that, by Claim~\ref{clm:nosandwich}, there is no sandwich of the
form \(rtr\) that blocks \(a\) in \(X\).  Therefore, \(r\) must follow \(at\) in \({\mS}\),
so \(atr \subseteq {\mS}\).  Because \(art \subseteq P\) and \(a\) is blocked by a \(d\)-sandwich, Lemma~\ref{lem:greedy3} (with \(r\) playing the role of \(b\)) implies that a subpattern of \(X\) is in \(\Gamma\), a contradiction.

\textbf{Case 2.3: \(e \neq r\) and \(e \neq t\)}.
Recall that we have \(e \in {\mS}\) and \(ea \subseteq {\mR}\).
Since \(e \in {\mS}\) and \(at \subseteq {\mS}\), 
one of \(eat\), \(aet\), or \(ate\) is in \({\mS}\).
We will first prove that \(e \notin P\) and \(e\) is not blocked by a sandwich
consisting only of colors in \(P\) in \(X\).

Assume first that \(e \in R\) and recall that 
\(r\) was defined to be the first color in \(R\) that is neither \(t\) nor \(a\).
Therefore, \(rea \subseteq {\mR}\), which implies that a sandwich of the form \(ere\)
blocks \(a\) in \(X \supseteq {\mS}d{\mR}\).
Note that this, with Claim~\ref{clm:nosandwich}, implies that \(e \notin P\).
Recall that Claim~\ref{clm:nosandwich} implies that every color in \(atP\) is unsortable
except \(a\). 
If \(e\) is blocked by a sandwich consisting of colors only in \(P\), then 
the subordering formed by removing all colors except
those in \(P\) and \(e\) from \(X - d\) is trivially unsortable.
This contradicts the assumption that all proper suborderings of \(X\) are sortable.

Now assume that \(e \notin R\). In this case, \(P \subseteq R\) implies that \(e \notin P\).
By the definition of \({\mR}\), the fact that \(e \in {\mR}\) implies that the last \(e\)
follows the last \(d\) in \(X\).
Since \(e \notin R\) and the state after \(d\) is sorted is \((S, R)\),
the color \(e\) was sorted before the color \(d\).
Since \(d\) was sorted before any color in \(P\),
the color \(e\) is not blocked by a sandwich consisting only of colors in \(P\) in \(X\).

So, in both cases, \(e \notin P\) 
and \(e\) is not blocked by a sandwich consisting only of colors in \(P\) in \(X\).
In particular, the subordering \(ata\) does not block \(e\) in \(X\).
This means that \(e\) does not follow \(a\) in \({\mR}\).
Note first that if \(eP \subseteq {\mR}\), then \ref{2sc} holds.
So we can assume \(eP \nsubseteq {\mR}\).
This, with the fact that \(e\) does does not follow \(a\) in \({\mR}\) and \(e \in {\mR}\), 
implies that \(P = rart\) and \(reart \subseteq {\mR}\).
Therefore, at least one of \(eatdreart\), \(aetdreart\), or \(atedreart\) is a subordering of \(X = {\mS}d{\mR}\).
This is a contradiction, because 
\(eatdreart \supseteq eatrear \cong dabtdat \cong I_{1,0}\),
\(aetdreart \supseteq aetreart \cong adbtdatb \cong \mathcal{C}_{2,4}\),
and
\(atedreart \supseteq atereart \cong abdtdatb \cong \mathcal{C}_{2,5}\).
\end{proof}

When Claim~\ref{clm:2stackfinal}\ref{2sa} holds, 
one of the orderings in the following table is a subordering of \(X\).
\begin{center}
    \begin{tabular}{|c||c||c||c|}
    \hline
\(datdarat \cong \mathcal{C}_{1,1}\)  & 
\(datdrart \cong \mathcal{C}_{1,7}\) &
\(datdarbrt \cong \mathcal{C}_{1,12}\) &
\(datdarcbct \subpatternr \mathcal{C}_{1,12}\) \\\hline

\(adtdarat \cong \mathcal{C}_{1,4}\) &
\(adtdrart \cong \mathcal{C}_{1,8}\) &
\(adtdarbrt \cong \mathcal{C}_{1,11}\) &
\(adtdarcbct \subpatternr \mathcal{C}_{1,11}\) \\\hline

\(atdedarat \cong \mathcal{C}_{1,5}\) &
\(atdedrart \cong \mathcal{C}_{1,9}\) &
\(atdedarbrt \cong \mathcal{C}_{1,13}\) & 
\(atdedarcbct \subpatternr \mathcal{C}_{1,13}\) \\\hline
\end{tabular}
\end{center}

If Claim~\ref{clm:2stackfinal}\ref{2sc} holds, then by swapping \(d\) and
\(e\), we can see that one of orderings in the preceding table is a subpattern of \(X\).
Finally, if Claim~\ref{clm:2stackfinal}\ref{2sb} holds, then 
one of the orderings in the following table is a subpattern of \(X\).
\begin{center}
    \begin{tabular}{|c||c||c||c|}
    \hline
\(atdtarat \cong \mathcal{C}_{1,2}\) &
\(atdatrat \cong \mathcal{C}_{1,3}\) &
\(atdtrart \cong \mathcal{C}_{1,6}\) &
\(atdtarbrt \cong \mathcal{C}_{1,10}\) \\\hline
\end{tabular}
\end{center}
\end{proof}

\section{Further Directions}

While the process of foot-sorting lends itself naturally to sorting with a stack, one could consider analogs of the sock sorting problem which utilize other stack configurations or other data structures entirely.

\subsection{Sorting with Multiple Stacks}

Defant and Kravitz introduced a notion of foot-sorting with multiple feet in \cite{defant2022footsorting}, though with more stacks in series, the sorting process becomes increasingly computationally intensive. It would be interesting to investigate perhaps the simplest extension of our work, the basis of sock orderings which are foot-sortable with two stacks in series. 

Defant and Kravitz proved in \cite{defant2022footsorting} that for any number \(n\geq 1\), one may construct a sock ordering which is unsortable with \(n\) stacks. Using this proof, one can check that the sock ordering of the form \((abcde)^{1440}\) (that is, the block \(abcde\), where \(a,b,c,d\) and \(e\) are distinct sock colors, repeated 1440 times) is unsortable with two stacks in series, but it might be challenging to construct a single sock ordering that is minimally unsortable in this setting.

\subsection{Deque Sorting}

Because a deque, or double-ended queue, allows for more access to the socks for addition and removal, one might think about deque-sorting for socks independently of a foot; a sock could be placed either over an outermost sock or within an innermost sock, and the same is possible for sock removal.\footnote{
Using a deque in the context of permutation sorting has been considered previously 
(e.g.\ see \cite{knuth97}, \cite{denton2012methods}, and \cite{price2017permutations}).} This modifies the original operations of foot-sorting, using the language of the deque, to be

\begin{enumerate}[left=40pt]
  \item[(\(\text{stack}_\uparrow\))] Take the leftmost sock on the right and place it on top of the deque (potentially over any socks that are already in the deque).
  \item[(\(\text{stack}_\downarrow\))] Take the leftmost sock on the right and place it on the bottom of the deque (potentially under any socks that are already in the deque).
  \item[(\(\text{unstack}_\uparrow\))] Remove the topmost sock from the deque and place it so that it becomes the rightmost sock on the left.
  \item[(\(\text{unstack}_\downarrow\))] Remove the bottommost sock from the deque and place it so that it becomes the rightmost sock on the left.
\end{enumerate}

While the addition of possible moves in this scenario makes for more difficult computation, some similar ideas hold as in the stack setting: in particular, it is now necessary to avoid the pattern \(abab\) from appearing in the deque, just as it is necessary to avoid \(aba\) appearing in the stack. It would be interesting to enumerate the basis of deque-sortable sock orderings, just as we have enumerated the basis of foot-sortable sock orderings.

It is straightforward to show that the basis of deque-sortable sock orderings is different from \(\Gamma\), as well as to show that this basis is infinite. Very closely mirroring the proof of Observation~\ref{obs:inf_unsort} under the relaxed restrictions of the deque, one can generate the class of sock orderings of the form \(dabtI_{1,n}\) for any \(n\geq 0\) and verify that such sock orderings are minimally unsortable with the deque. However, with the increased computational difficulty in the deque setting, we do not know whether other minimal deque-unsortable sock orderings may be generated from \(\Gamma\) in a similar fashion. 

\section{Acknowledgement}

We would like to thank  the anonymous referees for their careful reading of the manuscript
and for their valuable comments.

\appendix
\section{Algorithm}\label{sec:algorithm}

In this appendix, we present a recursive algorithm for foot-sorting that corresponds to our inductive proof 
of Theorem~\ref{thm:main}.
We present this algorithm only for expository purposes, 
and we made no attempt to optimize the running time of this algorithm.
(Note that, as referenced in the introduction, 
independent of our work, Yu has already described a fast algorithm \cite{yu2023deciding}.)

The algorithm first attempts to sort using the greedy algorithm.
\smallskip\hrule\smallskip
\noindent\textbf{Greedy algorithm} - \textit{Input}: A sock ordering \(X\).
\begin{enumerate}[label=(G.\arabic*)]
  \item While the current state is not terminal, do the following.
    \begin{enumerate}[label=(G.\arabic{enumi}.\arabic*)]
      \item Sort the sortable color which terminates first.
        That is, if there is a color that only appears on the top of the stack, then sort that color,
        and otherwise sort the sortable color on the right which terminates first. 
    \end{enumerate}
\end{enumerate}
\hrule\bigskip

The main algorithm also makes use of the following helper routine.
The input to the helper routine is a sock ordering \(X\) and colors \(d\) and \(a\)
where we assume the input satisfies the hypothesis of Lemma~\ref{lem:nodsandwich} 
in the following sense:
There is a way to sort \(X\) to a state \((S, R)\) so that
\begin{itemize}
  \item \(d\) was the color sorted immediately before entering the state \((S, R)\);
  \item for some substack \(S' \subseteq S\) and subordering \(R' \subseteq R\),
    the color \(a\) is beneath the top color on \(S'\) and
    the color \(a\) is the only good sortable color in \(S'R'\). (In particular,
    this implies that \((S', R')\), and hence \((S,R)\), is unsortable.)
\end{itemize}
When this routine is called, the colors \(a\) and \(d\) will also have the property that 
if \(a\) is blocked by a \(d\)-sandwich in \(X\)
(i.e.\ Lemma~\ref{lem:nodsandwich}\ref{I} holds), then 
(by either Lemma~\ref{lem:greedy3} or Lemma~\ref{lem:greedy2}) 
a subpattern of \(X\) is in \(\Gamma\), so \(X\) is unsortable.

\smallskip\hrule\smallskip
\noindent\textbf{Helper routine} - \textit{Input}: A sock ordering \(X\) and colors \(d\) and \(a\).
\begin{enumerate}[label=(H.\arabic*)]
  \item If \(a\) is blocked by a \(d\)-sandwich, then \(X\) is unsortable.
  \item Otherwise, we recursively run the main algorithm on the smaller ordering \(X - d\).
  \item If it fails, then by our recursive assumption \(X - d\) is unsortable, so \(X\) is unsortable.
  \item Otherwise, \(X - d\) is sortable and we have the following two possibilities.
    \begin{enumerate}[label=(H.\arabic{enumi}.\arabic*)]
      \item When we attempt to sort \(X\) using the same sequence of colors that fully 
        sorted \(X - d\),
        there is a stage in which the color \(d\) only appears as the top color of the stack, which
        means that we can fully sort \(X\) by sorting \(d\) at this stage and 
        then continuing with the sequence that was used to sort \(X - d\).
      \item \(X\) is unsortable because \(I_{1,n} \in \Gamma\) is a subpattern of \(X\).
    \end{enumerate}
\end{enumerate}
\hrule\bigskip

We now present the main algorithm.
It closely follows the proof Theorem~\ref{thm:main}.

\smallskip\hrule\smallskip
\noindent\textbf{Main Algorithm} - \textit{Input}: A sock ordering \(X\).
\begin{enumerate}[label=(M.\arabic*)]
  \item Let \((S,R)\) be the output of the greedy algorithm performed on \(X\).
  \item If the greedy algorithm succeeds (i.e.\ \(S = R = \emptyset\)), then we are done.
  \item If \(S = \emptyset\) or there is a color \(x \in S\) on the stack such that \((x, R)\) is terminal,
    then \(X\) has a trivially unsortable subpattern, so \(X\) is unsortable.
  \item Otherwise, 
    there are at least two colors on the stack, so let \(t\) be the top
    color and let \(b\) be the color immediately beneath it on the stack.
    We can also let \(d\) be the final color sorted by the greedy algorithm.
  \item If \((bt, R)\) is terminal, then we can finish by running the helper routine with input
    \(X\), \(d\), and \(b\) (c.f.\ Lemma~\ref{lem:greedy2} and Case 2 in Theorem~\ref{thm:main}).
  \item Otherwise, let \(a\) be the lowest color on the stack
    such that \((abt, R)\) is terminal. (Such a color exists by Lemma~\ref{lem:greedy4}.)
    We must have \( at \subseteq R \). 
    (This follows from Claim~\ref{clm:asinR} in the proof of Theorem~\ref{thm:main}.)
  \item If there exists a color \(z\) that is both in \(R\) and 
    above \(a\) and below or equal to \(b\) on the stack \(S\), then do the following (c.f.\ Lemma~\ref{lem:greedy3} and Case 3.1 in Theorem~\ref{thm:main}).
    \begin{enumerate}[label=(M.\arabic{enumi}.\arabic*)]
      \item If \(azt \subseteq R\) or \(atz \subseteq R\), then we can finish by running the helper routine
        with input \(X\), \(d\), and \(a\). 
      \item Otherwise, \(zat \subseteq R\), and we can finish by running the helper routine with 
        input \(X\), \(d\), and \(z\).
    \end{enumerate}
  \item Otherwise, we have \(b \notin R\), so if we backtrack to
    the state \((S', R')\) immediately before the final \(b\) was pushed onto the stack
    the following holds:
    \begin{itemize}
      \item There exists a string \(C\) with \(b \notin C\) and strings \(A\) and \(A'\) such that
        \(X = AbC\) and \(R' = A'bC\) and \(S'A' \subseteq A\).
      \item The next color sorted by the greedy algorithm (i.e.\ the color that
        pushes the final \(b\) onto the stack) is a color \(c \in C\). 
      \item For some nonnegative integer \(n\), there are \(n+1\) distinct colors
        \(t_0, \dotsc, t_n\) that are distinct from \(b\), \(c\) and \(a\),
        such that \(t_0cI(t_0, \dotsc, t_n)at_n \subseteq C\).
        (This follows from Lemma~\ref{lem:interlinking}.)
    \end{itemize}
  \item Therefore, if \(ca \subseteq A\), then \(I_{1,n} \in \Gamma\) is a subpattern of \(X\),
    so \(X\) is unsortable.
  \item\label{alg:Sprime1} Otherwise, there are at least three distinct colors on the stack \(S'\).
    (This follows from Claim~\ref{clm:Sprime} in the proof of Theorem~\ref{thm:main}.)
  \item In particular, this implies that there is a color \(d'\) that was sorted immediately 
    before the greedy algorithm reached the state \((S', R')\).
  \item\label{alg:Sprime2} Furthermore, if we let \(y\) be the color immediately beneath the top color on the 
    stack \(S'\), we can finish by running the helper routine with input \(X\), \(d'\), and \(y\)
    (c.f.\ Lemma~\ref{lem:greedy3} and Claim~\ref{clm:Sprime} 
    in the proof of Theorem~\ref{thm:main}).
\end{enumerate}
\end{document}